\documentclass{amsart}

\usepackage{graphicx} \usepackage{tikz-cd} 

\usepackage{tikzit}

\tikzstyle{red dot}=[fill={rgb,255: red,255; green,37; blue,102}, draw=black, shape=circle]
\tikzstyle{green}=[fill={rgb,255: red,81; green,255; blue,7}, draw=black, shape=circle, tikzit category=green dot]
\tikzstyle{orange}=[fill={rgb,255: red,255; green,128; blue,0}, draw={rgb,255: red,255; green,128; blue,0}, shape=circle]
\tikzstyle{For labels}=[fill=white, draw=red, shape=circle]
\tikzstyle{FOR LABEL-GREEN}=[fill=white, draw={rgb,255: red,28; green,179; blue,11}, shape=circle]
\tikzstyle{blue}=[fill=blue, draw=blue, shape=circle]
\tikzstyle{Rectangle}=[fill=white, draw=red, shape=rectangle]
\tikzstyle{black}=[fill={rgb,255: red,8; green,8; blue,8}, draw=none, shape=circle]
\tikzstyle{circle}=[fill=white, draw=black, shape=circle]

\tikzstyle{new edge style 0}=[->]
\tikzstyle{new edge style 1}=[draw=red, -]

\usepackage{amsmath, amssymb, amsthm}  
\usepackage[pdftex,  colorlinks=true]{hyperref} 

\usepackage{xcolor} 

\usepackage{yhmath}

\usepackage{accents}
 
\newtheorem{theorem}{Theorem}[section]
\newtheorem{lemma}[theorem]{Lemma}
\newtheorem{proposition}[theorem]{Proposition}
\newtheorem{corollary}[theorem]{Corollary}

\theoremstyle{definition}
\newtheorem{remark}[theorem]{Remark}
\newtheorem{example}[theorem]{Example}
\newtheorem{question}[theorem]{Question}

\theoremstyle{definition}
\newtheorem{definition}[theorem]{Definition}

\DeclareMathOperator{\comm}{\mathsf{Comm}}

\DeclareMathOperator{\dist}{\mathsf{dist}}
\DeclareMathOperator{\Hdist}{\mathsf{hdist}}

\newcommand{\NN}{\mathbb{N}}
\newcommand{\ZZ}{\mathbb{Z}}

\newcommand{\RR}{\mathbb{R}}

\newcommand{\calp}{\mathcal{P}}
\newcommand{\calP}{\mathcal{P}}

\newcommand{\calq}{\mathcal{Q}}
\newcommand{\calr}{\mathcal{R}}
\newcommand{\cals}{\mathcal{S}}

\newcommand{\mc}{\mathcal}

\newcounter{ccomments}

\newcounter{ecomments}

\begin{document}

\title{The quasi-isometry invariance of the Coset Intersection Complex}

\author[ ]{Carolyn Abbott } 
\address{Brandeis University, Waltham, MA, USA}
\email{carolynabbott@brandeis.edu}

\author[]{Eduardo Mart\'inez-Pedroza }  
\address{Memorial University of Newfoundland, St. John's, NL, Canada}
\email{emartinezped@mun.ca}

\date{\today}
 
\maketitle

\begin{abstract}
For a pair $(G,\calp)$ consisting of a group and finite collection of subgroups, we introduce  a simplicial $G$-complex $\mathcal{K}(G,\calp)$ called the coset intersection complex. We prove that the quasi-isometry type and the homotopy type of  $\mathcal{K}(G,\calp)$ are quasi-isometric invariants of the group pair $(G,\calp)$. Classical  properties of $\calp$ in $G$  correspond to topological or  geometric properties of $\mathcal{K}(G,\calp)$, such as  having finite height,  having finite width,  being almost malnormal, admiting a malnormal core, or having thickness of order one.  As applications, we obtain that a number of algebraic properties of $\calp$ in $G$ are quasi-isometry invariants of the pair $(G,\calp)$.  For a certain class of right-angled Artin groups and their maximal parabolic subgroups, we show that the complex $\mc K(G,\calp)$ is quasi-isometric to the Extension graph; in particular, it is quasi-isometric to a tree. 
\end{abstract}

\tableofcontents

 \section{Introduction}

A \emph{group pair} $(G,\mathcal P)$ is a pair consisting of a finitely generated group $G$  and  a finite collection $\mathcal P$ of  infinite  subgroups. The collection $\mc P$ is called the \textit{peripheral structure} of the pair. A map $q\colon (G,\mathcal P) \to (H,\mathcal Q)$ is a quasi-isometry of group pairs if $q\colon G \to H$ is a quasi-isometry such that for each $P\in \mathcal P$, the image $q(P)$ is uniform finite Hausdorff distance from a left coset of some $Q\in\mc Q$ and the quasi-inverse of $q$ satisfies the analogous property; see Definition~\ref{defn:quasi-isometry-pairs}.

Quasi-isometries of group pairs have been used implicitly in many works \cite{DS05, BDM09, MSW11, LFS15, KaLe97}, typically to prove strong rigidity results about a group.  For example, Kapovich and Leeb used similar ideas to prove the quasi-isometry invariance of the geometric decomposition of Haken manifolds~\cite[\S 5.1]{KaLe97}, as did Drutu and Sapir and Behrstock, Drutu, and Mosher to prove quasi-isometric rigidity results for relatively hyperbolic groups \cite{DS05, BDM09}.  
 More recently, quasi-isometries of group pairs have been used explicitly by a number of  researchers in group theory; see, e.g., \cite{BuHr21, HaHr19, GT21,  EMP21-qichar, EMP27-HypEmb-Sam, Ge19}, and~\cite{HMS2021} for a recent survey.

 This article introduces the \emph{coset intersection complex} of a group pair $(G,\calP)$, which we denote $\mc K(G,\calP)$. It is a simplicial complex whose quasi-isometry type and homotopy type are quasi-isometry invariants of the group pair (Theorem~\ref{thm:SimplicialMap}).  We reformulate several natural algebraic properties of group pairs into geometric and combinatorial properties of $\mc K(G,\calP)$ (Theorem~\ref{thm:alg&geom}).  This reformulation allows us to prove that certain natural finiteness properties of group pairs are quasi-isometry invariants (Theorem~\ref{thm:GeomProps}). In this way, the coset intersection complex reinforces the central paradigm of geometric group theory: it is fruitful to study algebraic properties from geometric and topological perspectives. 

\begin{definition}[Coset intersection complex $\mathcal{K}(G,\calp)$] Let $(G,\calp)$ be a group pair, and let $G/\calP$ be the set of all cosets $gP$ with $g\in G$ and $P\in\calP$. The \emph{coset intersection complex} is the  simplicial $G$-complex with vertex set $G/\calp$ and such that a set $\{g_1P_1,\ldots ,g_kP_k\}\subset G/\calp$ defines a  simplex if $\bigcap_{i=1}^k g_iP_ig_i^{-1}$ is an infinite subgroup of $G$. Observe that the $G$-action on $G/\calp$ by left multiplication  extends to a cellular $G$-action on $\mathcal{K}(G,\calp)$. 
\end{definition}

  For example, if $G$ is hyperbolic relative to a collection of subgroups $\calp$, then $\mc K(G,\calp)$ is a discrete graph.  At the opposite extreme, if $G$ is abelian and $\calp=\{P\}$ is a single infinite subgroup, then $\mc K(G,\calp)$ is a complete graph. We give several  examples of coset intersection complexes for particular group pairs in Example~\ref{ex:BasicExs}. 
  
  We mention the more involved example of right-angled Artin groups here.  We show in Section~\ref{sec:RAAGs} that the coset intersection complex of certain right-angled Artin groups is related to several well-known graphs, including the  \emph{contact graph},   introduced by Hagen~\cite{HagenThesis}, and the \emph{extension graph}, introduced by Kim and Koberda.  Hagen showed that the contact graph is quasi-isometric to a tree~\cite{HagenThesis}, and, under certain assumption, the extension graph is quasi-isometric to the contact graph~\cite{KK13}. 
   
  
  \begin{theorem} \label{thm:RAAG}   Suppose $G$ is a right-angled Artin group whose defining graph $\Gamma$ is connected, triangle-free, and has no vertices of valence $<2$.  Let $\calp$ be the collection of maximal standard abelian subgroups of $G$.  Then  $\mc K(G,\calp)$ is quasi-isometric to the extension graph and the contact graph.  
  \end{theorem}

\begin{corollary}
    Under the assumptions of Theorem~\ref{thm:RAAG}, the complex $\mc K(G,\calp)$ is hyperbolic and quasi-isometric to a tree. 
\end{corollary}

Given a quasi-isometry of group pairs $q\colon (G,\calp) \to (H,\calq)$, there is an induced map $\dot q\colon \mc K(G,\calp) \to \mc K(H,\calq)$. Roughly, the image of a vertex $gP$ in $\mc K(G,\calp)$ is a vertex $hQ$ in $\mc K(H,\calq)$ such that $hQ$ and $q(gP)$ are at finite Hausdorff distance. The map $\dot q$ is not unique: there may be many cosets $hQ$ at finite Hausdorff distance from $q(gP)$.  Any choice of such $hQ$ will result in an induced map; by an abuse of notation, we denote any such induced map by $\dot q$.  The following theorem shows that the induced maps $\dot q$ have particularly nice properties.

\begin{theorem}
\label{thm:SimplicialMap}
Let $q\colon (G,\calp) \to (H,\calq)$ be a quasi-isometry of group pairs and $\dot q\colon \mc K(G,\calp) \to \mc K(H,\calq)$ an induced map. 
\begin{enumerate}
\item \label{item:sxmap} The induced map $\dot q$ is a simplicial map and a homotopy equivalence.

\item \label{item:htpcmaps} All choices of $\dot q$ are homotopic as continuous maps $\mathcal{K}(G,\calp) \to \mathcal{K}(H,\calq)$.

\item \label{item:1skelQI} The induced map $\dot q\colon \mathcal{K}(G,\calp) \to \mathcal{K}(H,\calq)$ is a quasi-isometry of simplicial complexes.  
\end{enumerate}
\end{theorem}

By a \emph{simplicial map $f\colon X \to Y$} we mean a  function between the sets of vertices of the simplicial complexes $X$ and $Y$ with the property that the images of the vertices of a simplex always span a simplex.  The quasi-isometry in \eqref{item:1skelQI} is with respect to a natural length metric on the complexes viewed as piecewise Euclidean complexes; see the discussion before Corollary~\ref{cor:QIofComplexes} for details. We also give conditions on the group pairs $(G,\calp)$ and $(H,\calq)$ to ensure that $\dot q$ is a simplicial isomorphism; see Proposition~\ref{prop:IsomorphicCosetIntersectionCmplex2}.    
 
We now turn our attention to properties of group pairs satisfying strong rigidity properties in relation with the coset intersection complex.  A property of group pairs is called \emph{geometric} if  every group pair  quasi-isometric to a group pair satisfying the property is virtually isomorphic to a group pair satisfying the property.  Here, a \textit{virtual isomorphism} is a generalization of commensurability to group pairs; see Definition~\ref{def:VirtIsomPair}. Let us note that if a property of group pairs is preserved by quasi-isometries then it is geometric, but the converse does not hold, as some of the properties studied in this article illustrate. 

There are a number of geometric properties of group pairs $(G,\calp)$ that have been studied in the literature, in some cases implicitly. Examples include
having a well-defined relative Dehn function~\cite{EMP23-QI-Dehn, Osin06}; having an almost malnormal peripheral structure~\cite{EMP23-QI-Dehn};  being relatively hyperbolic with respect to the peripheral structure~\cite{DS05,BDM09};  and  having a finite number of filtered ends (Bowditch's coends) with respect to the peripheral structure~\cite{EMP21-qichar}.   The coset intersection complex encodes several  algebraic finiteness  properties of the group which turn out to be geometric.

\begin{theorem}\label{thm:GeomProps}
 The following define geometric properties of group pairs $(G,\calp)$:
\begin{enumerate}
    \item \label{item:finiteheight} $\calp$ has finite height in $G$.
    \item \label{item:width} $\calp$ has finite width in $G$.

    \item \label{item:fencedint} $\calp$ has fenced infinite intersections in $G$.


    \item \label{item:malnormal} $\calp$ is an almost malnormal collection in $G$.
    \item \label{item:network} $G$ is a network with respect to $\mathcal P$.

    \item \label{item:thickness} $G$ is thick of order $1$.  
    
    \item \label{item:reduced} Every $P\in\calp$ satisfies $P=\comm_G(P)$.   

    \item \label{item:reducible} Every $P\in\calp$ has finite index in $\comm_G(P)$.

    \item \label{item:boundedpacking} $\calp$ has bounded packing in $G$. 

    \item
    \label{item:finitepacking} $\calp$ has finite packing in $G$.
\end{enumerate}   
\end{theorem}

We discuss each of these algebraic properties in more depth in the main body of the article, where  we also recall their definitions.  Here, we give a brief overview of where they appear in the literature.  The study of \textit{height} and \textit{width} of subgroups can be traced back to the work of Gitik, MJ, Rips and Sageev~\cite{GMJRS98} in the late nineties,  where they proved  that quasi-convex subgroups of hyperbolic groups have finite width and, in particular, finite height. It is a well-known open question of Swarup whether, in a hyperbolic group,  a subgroup with finite height and finite index in its  commensurator is  quasiconvex~\cite[Question 1.8.]{BestvinaQuestions}. Parts \eqref{item:finiteheight} and \eqref{item:boundedpacking} of Theorem~\ref{thm:GeomProps}  could be regarded as evidence of a positive answer to Swarup's question in the sense that the geometry of such subgroup in the ambient group is preserved by quasi-isometry. The literature on the study of height and width of subgroups also includes the work of Hruska and Wise~\cite{HW09} on finite width of strongly quasi-convex subgroups of relatively hyperbolic groups, and  the  work of Antolin, MJ, Sisto and Taylor  on finite width of stable subgroups~\cite{AMJST19}. \textit{Fenced infinite intersections} for group pairs involves uniform bounds on the distance between elements of $G/\calp$ when the intersection of the corresponding conjugates is infinite; see Definition~\ref{def:FencedInfInt}.   This notion has appeared implicitly in the study of non-positively curved groups in relation to height and width, for example,~\cite[Lemma 1.2]{GMJRS98} and
~\cite[Lemma 8.4]{HW09}. 
The \emph{commensurated core} of a group pair is a  generalization of the notion of a malnormal core.  The latter was introduced  by Agol, Groves and Manning in the context of group pairs $(G,P)$ where $G$ is a hyperbolic group and $P$ is a quasi-convex subgroup~\cite{AGM09}.  
\textit{Amost malnormality} has been studied in many contexts, including relative Dehn functions of group pairs,  peripheral subgroups in relatively hyperbolic groups and, more generally, hyperbolically embedded subgroups in acylindrically hyperbolic groups \cite{Osin06, DGOsin2017}. 
That almost malnormality is a geometric property of group pairs is a result of Hughes, Sanchez-Salda\~na and the second  author~\cite[Theorem D]{EMP23-QI-Dehn}; in this article we provide an alternative argument.
\textit{Networks} were introduced by Behrstock, Dru\c{t}u, and Mosher
in connection with the notion of \textit{thickness} of a group, which is an obstruction to relative hyperbolicity 
~\cite{BDM09}. \textit{Bounded packing} was introduced by Hruska and Wise  in relation to the study of width and height of subgroups and actions on $\mathsf{CAT}(0)$ cubical complexes~\cite{HW09}. The notion of \emph{finite packing} is introduced in this article; see Definition~\ref{def:NewFinitePacking}.  It relates to height and bounded packing, and is equivalent to cocompactness of the coset intersection complex.

Each part of Theorem~\ref{thm:GeomProps} is proven by first using Theorem~\ref{thm:alg&geom}, below, to translate the algebraic property into a geometric property of an associated (extended) metric space, and then using the geometry of that space to prove that the property is geometric.  For most of the properties, the space is the coset intersection complex $\mc K(G,\calp)$.  For parts \eqref{item:malnormal}, \eqref{item:reducible}, and \eqref{item:boundedpacking}, though, we use a different  space, which we call the \textit{coset space}.  This is the extended metric space $(G/\calp, \text{hdist}_G)$, where $\text{hdist}_G$ is the Hausdorff distance on subsets of $G$ with respect to a word metric.

 \begin{theorem}\label{thm:alg&geom}
Let $(G,\calp)$ be a group pair and $\mathcal{K}(G,\calp)$ the coset intersection complex. 
\begin{enumerate}
    \item \label{item:finitediml} $\mathcal{K}(G,\calp)$ is finite dimensional if and only if $\calp$ has finite height in $G$.

    \item \label{item:cliques} The cardinality of cliques in   $\mathcal{K}(G,\calp)$ is   bounded if and only if $\calp$ has finite width in $G$. 

    \item \label{item:Gcocompact1skel} $\mathcal{K}(G,\calp)$ has $G$-cocompact $1$-skeleton if and only if $\calp$ has fenced infinite intersections in $G$.
        
    \item\label{item:Gcocompact}  $\mathcal{K}(G,\calp)$ is $G$-cocompact if and only if   $(G,\calp)$ has finite height and a finite commensurated   core.  

    \item \label{item:0diml} $\mathcal{K}(G,\calp)$ is zero dimensional if and only if $\calp$ is   almost malnormal in $G$.
    \item \label{item:connected} 
      $\calp$ is non-empty and $\mathcal K(G,\calp)$ has a connected and $G$-cocompact subgraph with vertex set $G/\calp$ if and only if $(G,\dist_G)$ is a network with respect to the collection $G/\calp$.
\end{enumerate}
\end{theorem}



From our perspective, there are two natural classes of groups/group pairs that merit deeper study.  First, one could consider group pairs for which the coset intersection complex $\mc K(G,\calp)$ is a quasi-isometry invariant of the \textit{group}, rather than of the group pair.  This will be the case, for example, when the collection of subgroups $\calp$ is qi-characteristic in $G$.  A collection of infinite subgroups $\calp$ of a finitely generated group $G$ is \emph{qi-characteristic} if the collection of subspaces $G/\calp$ is preserved up to uniform finite Hausdorff distance by every quasi-isometry of $G$. Containing a finite collection of qi-characteristic subgroups is a quasi-isometry invariant of finitely generated groups~\cite[Theorem 1.1]{EMP21-qichar}. Specifically, if $(G,\calp)$ is a group pair and 
$H$ is a finitely generated group quasi-isometric to $G$, then there is a qi-characteristic collection of subgroups $\calq$ of $H$ such that the group pairs $(G,\calp)$ and $(H,\calq)$  are quasi-isometric pairs.  The class of finitely generated groups that contain a
finite collection of qi-characteristic subgroups is large and includes hyperbolic groups relative to non-relatively hyperbolic groups~\cite{BDM09}, atomic right-angled Artin groups~\cite{BKS08}, lamplighter-like groups~\cite{GT21, GT24}, and mapping class groups~\cite{BKMM12}. Theorem~\ref{thm:SimplicialMap} implies that the (quasi-isometry class of the) coset intersection complex $\mathcal{K}(G,\calp)$ of a group pair $(G,\calp)$ with $\calp$ a qi-characteristic collection is a quasi-isometry invariant of the group $G$. 

Second, since actions on hyperbolic spaces have proven to be  a powerful tool in geometric group theory, one could consider group pairs $(G,\calp)$ such that $\mc K(G,\calp)$ is hyperbolic.  Theorem~\ref{thm:RAAG} provides one example of such a group pair, but there are certainly others.  In fact, the collection of maximal abelian subgroups $\calp$ in atomic right-angled Artin groups (which satisfy the conditions of Theorem~\ref{thm:RAAG}) are  qi-characteristic \cite[Theorems 1.1 \& 1.6]{BKS08}, and so these groups fall into the intersection of the two classes.  In this case, the Gromov boundary of $\mc K(G,\calp)$ is a quasi-isometry invariant of the group $G$.

 \begin{question}
What are  examples of finitely generated groups (not commensurable to right-angled Artin groups) that have a collection of qi-characteristic subgroups with respect to which the coset intersection complex is hyperbolic?
 \end{question}

For general right-angled Artin groups and generalizations, the coset intersection complex is closely related to several well-known complexes, including the extension graph of Kim and Koberda~\cite{KK13}, the extension complex of Huang~\cite{Hu17}, the crossing complex associated to a quasi-median graph of Genevois~\cite{Ge17}, and the intersection complex of associated to a  weakly special square complex of Oh~\cite{Oh22,Oh23}. In particular, the work of Huang~\cite{Hu17} and of Oh~\cite{Oh22} contain versions of Theorem~\ref{thm:SimplicialMap} in the sense that quasi-isometries between groups in the pertinent class induce isomorphisms of the respective complexes. 
In the last section of the article, we give an account of some of the known  relations   between the mentioned complexes.

\subsection*{Organization} 

Section~\ref{sec:QIofPairs} provides background on and basic properties of quasi-isometries of group pairs.  In Section~\ref{sec:CosetSpace}, we introduce the coset space associated to a group pair and use it to prove Theorem~\ref{thm:GeomProps} \eqref{item:malnormal}, \eqref{item:reduced}, \eqref{item:reducible}, and \eqref{item:boundedpacking}.

In Section~\ref{sec:CosetIntCx}, we introduce the coset intersection complex.  We give some basic examples and properties of the complex, proving Theorem~\ref{thm:alg&geom} \eqref{item:finitediml}, \eqref{item:cliques}, and \eqref{item:0diml}.  We then show that a quasi-isometry of pairs induces a simplicial map of the associated coset intersection complexes and prove Theorem~\ref{thm:SimplicialMap}, which describes several properties of this induced map.   

In Section~\ref{Sec:GeomProps}, we translate several algebraic finiteness conditions into geometric properties of the coset intersection complex and prove these properties are geometric.  This proves the remaining parts of Theorems~\ref{thm:GeomProps} and \ref{thm:alg&geom}.

Finally, in Section~\ref{sec:RAAGs}, we give a detailed example of the coset intersection complex for certain right-angled Artin groups and discuss the relationship between this complex and other complexes associated to the group, proving Theorem~\ref{thm:RAAG}.

\subsection*{Acknowledgements:} The authors thank   Sangrok Oh, and Russ Woodroofe for comments and feedback on the first version of the article.  Special thanks to Antony Genevois for discussion on the topics of the article and for pointing out  a necessary correction in a first version of the article. 
The authors thank the anonymous referee for comments and suggestions 
that improved the exposition of the paper.
The first author was  supported by NSF grant DMS-2106906.
The second author thanks the first author for her hospitality while visiting Brandeis University where the ideas of this article originated. Parts of this article were written while the second author was visiting Tokyo Metropolitan University and Universidad Nacional de Colombia while on sabbatical; the second author thanks Tomohiro Fukaya and Mario Velasquez for their hospitality.  The second author acknowledges funding by the Natural Sciences and Engineering Research Council of Canada NSERC.

\section{Quasi-isometries of pairs}\label{sec:QIofPairs}
We begin by setting up some language and notation for the rest of the article. A \emph{pseudo-metric} is a metric such that the distance between two distinct points can be zero. 
An \emph{extended (pseudo-)metric} on a set $X$ is a function $\dist\colon X \times X \to [0,\infty]$   that satisfies the usual properties of a (pseudo-)metric  with  the triangle inequality interpreted in the natural way.  In particular, we consider $[0,\infty]$ with the standard order where  $\infty$ is the maximum value. A set together with an extended metric is called an \emph{extended metric space}. 

Let $(X,\dist_X)$ be an extended metric space. For a subset $A\subseteq X$, its closed $r$-neighborhood is denoted as
\[ \mathcal{N}_r(A)=\{x\in X \mid \text{$\exists a\in A$ such that $\dist_X(a,x)\leq r$}\}.\]
The Hausdorff distance between subsets $A,B\subseteq X$ is defined as 
\[ \Hdist_X(A,B)= \inf\{ r \mid B\subseteq \mathcal{N}_r(A) \text{ and } A\subseteq \mathcal{N}_r(B)\}, \]
where, as a convention, the infimum of the empty set is defined as infinity. An important property of the Hausdorff distance is that it satisfies the triangle inequality, in the sense that, if $A,B,C$ are subsets of $X$ then 
\[\Hdist_X(A,C) \leq \Hdist_X(A,B)+\Hdist_X(B,C).\]
In particular, the power set of $X$ with the Hausdorff distance, $(\mathcal{P}(X), \Hdist_X)$ is an extended pseudo-metric space.  
The infimum distance between subsets is 
\[ \dist_X(A,B) = \inf\{ \dist_X(a,b) \mid a\in A \quad b\in B\} .\]
Note that the infimum distance between sets does not satisfy the triangle inequality in general.


The notion of  quasi-isometry $f\colon X \to Y$ between metric spaces naturally  generalizes to  extended pseudo-metric spaces.

\begin{definition}[Quasi-isometry]\label{def:quasi-isometry}
Let $X$ and $Y$ be extended pseduo-metric spaces, and let $L\geq1$ and $C\geq0$ be real numbers. A function $f\colon X\to Y$ is an $(L,C)$-quasi-isometry if  
 for any $a,b\in X$, 
\begin{enumerate}
   \item $\dist_X(a,b)=\infty$ if and only if $\dist_Y(f(a),f(b))=\infty$;
    \item if $\dist_X(a,b)<\infty$, then 
    \[\frac{1}{L}\dist_X(a,b)-C\leq \dist_Y(f(a),f (b))\leq L\dist_X(a,b)+C; \]
    \item for every $y\in Y$ there is $x\in X$ such that $\dist_Y(f(x),y)\leq C<\infty$.
\end{enumerate} 
Two extended pseudo-metric spaces are \emph{quasi-isometric} if there is a quasi-isometry between them. As in the case of metric spaces, it is straight forward to verify that   if $f\colon X \to Y$ is quasi-isometry of extended pseudo-metric spaces, then there is a quasi-isometry $g\colon Y\to X$ such that $g\circ f$ and $f\circ g$ are uniformly close to the identity functions on $X$ and $Y$ respectively; such function $g$ is called a \emph{quasi-inverse} of $f$. 
\end{definition}

\begin{definition}[Uniformly finite-to-one function]
 A function is \emph{uniformly finite-to-one} if there is a  uniform finite bound on the cardinality of the pre-image of every element of the codomain.   
\end{definition}

\begin{definition}[Uniformly locally finite metric space]
 A metric space is \emph{uniformly locally finite} if for each $r$ there is a uniform bound of the cardinality of balls of radius $r$.    
\end{definition}

The notion of a quasi-isometry of pairs was introduced in \cite{EMP21-qichar}.
\begin{definition}[Quasi-isometry of pairs]\label{defn:quasi-isometry-pairs} 
Let $X$ and $Y$ be metric spaces, and let $\mathcal{A}$ and $\mathcal{B}$ be   collections of subspaces of $X$ and $Y$ respectively.  A  quasi-isometry $q\colon X\to Y$ is a \emph{quasi-isometry  of pairs} $q\colon (X,\mathcal{A}) \to (Y,\mathcal{B})$ if there is $M>0$ such that:
\begin{enumerate}
    \item for any $A\in \mathcal{A}$, 
    the set $\{ B\in \mathcal{B} \colon \Hdist_Y(q(A), B) <M \}$ is non-empty; and       
    \item for any $B\in \mathcal{B}$, 
    the set $\{ A\in \mathcal{A} \colon \Hdist_Y(q(A), B) <M \}$ is non-empty. 
\end{enumerate}
In this case, if $q\colon X \to Y$ is an $(L,C)$-quasi-isometry, then $q\colon (X, \mathcal{A})\to (Y,\mathcal{B})$ is called a \emph{$(L,C,M)$-quasi-isometry of pairs}. If there is a quasi-isometry of pairs  $(X,\mathcal{A}) \to (Y,\mathcal{B})$ we say  that $(X,\mathcal{A})$ and  $(Y,\mathcal{B})$ are \emph{quasi-isometric pairs}. One can verify that if $f\colon (X, \mathcal A) \to (Y, \mathcal B)$ is quasi-isometry of pairs and $g$ is a quasi-inverse of $f\colon X \to Y$ then $g\colon (Y,\mathcal B) \to (X, \mathcal A)$ is a quasi-isometry of pairs. 
\end{definition}

Given a quasi-isometry of pairs, there is an induced quasi-isometry on the corresponding sets of subspaces equipped with the  Hausdorff distance.
\begin{proposition}[Induced quasi-isometry $\dot q$]\label{propF:InducedDot-q-03}
If $q\colon (X, \mathcal A) \to (Y, \mathcal B)$ is an  $(L,C,M)$-quasi-isometry of pairs, then there is a quasi-isometry of extended metric spaces
\[ \dot q\colon (\mathcal{A},\Hdist_X) \to (\mathcal{B},\Hdist_Y)\]
such that
\begin{enumerate}
    \item $\Hdist_Y(\dot q(A), q(A))\leq M$  for every $A\in \mathcal{A}$. 
    \item If $(\mathcal A, \Hdist_X)$ is (uniformly) locally finite, then  $\dot q$ is (uniformly)  finite-to-one.
    \item If $p\colon (Y,\mathcal B) \to (X,\mathcal A)$ is a quasi-inverse of the quasi-isometry of pairs $q$, then $\dot p$ is a quasi-inverse of $\dot q$.
    \item For every $A_1$ and $A_2$ in $\mathcal A$,
    \[\dist_Y(\dot q(A_1) , \dot q(A_2) ) \leq L\dist_X(A_1, A_2)+C+2M\]
    and \[    \dist_X(A_1, A_2)  \leq   L \dist_Y(\dot q(A_1) , \dot q(A_2) ) +  (C+2M)L.\]  
\end{enumerate}
\end{proposition}
\begin{proof}The definition of quasi-isometry of pairs implies that there is a function $\dot q\colon \mathcal{A} \to \mathcal{B}$  such that $\Hdist_Y(\dot q(A), q(A))\leq M$ for every $A\in \mathcal{A}$. It follows that 
\[ \frac{1}{L}\Hdist_X(A_1,A_2)-C-2M\leq \Hdist_Y(\dot q(A_1),\dot q(A_2))\leq L\Hdist_X(A_1,A_2)+C+2M. \]
On the other hand, for every $B\in \mathcal B$ there is $A\in \mathcal{A}$ such that $\Hdist_Y(\dot q(A), B) \leq 2M$. 
Hence $\dot q\colon (\mathcal{A}, \Hdist_X)\to (\mathcal{B}, \Hdist_Y)$ is an $(L,C+2M)$-quasi-isometry. For the second statement, note that if $\dot q(A_1) = \dot q(A_2)$ then 
$\Hdist_X(A_1, A_2) \leq  L (2M+C)$; hence, if $(\mathcal{A}, \Hdist_X)$ is (uniformly )locally finite, then  $\dot q$ is (uniformly) finite to one. 
For the third statement, by increasing $M$, suppose that \[ \Hdist_X(\dot p(B) , p(B))\leq M, \quad   \Hdist_X(p\circ q (x) , x)\leq M,\quad \text{and} \quad  \Hdist_Y(q\circ p (y) , y)\leq M \]  for all $ B\in  \mathcal{B}$, $x\in X$ and $y\in Y$. Hence if  $A\in \mathcal{A}$ then 
\[ 
\begin{split}
\Hdist_X(\dot p \circ \dot q (A), A ) & \leq
\Hdist_X(\dot p \circ \dot q (A), p\circ \dot q (A) )+\Hdist_X(p\circ \dot q (A), A ) \\
& \leq M +\Hdist_X(p\circ \dot q (A), A ) \\
& \leq M + \Hdist_X(p\circ \dot q (A), p\circ q (A)) + \Hdist_X(p\circ  q (A), A ) \\
& \leq M + L\Hdist_Y(\dot q(A), q(A)) + C  + M \\
& \leq 2M+ LM + C=: M'.
\end{split}
\]
Analogously, $\Hdist_Y(\dot q \circ \dot p (B), B )\leq M'$ for every $B\in \mathcal{B}$. Hence $\dot p$ and $\dot q$ are quasi-inverses. The last statement is also a direct verification and  is left to the reader.
\end{proof}

\begin{remark}
 The quasi-isometry $\dot q\colon (X,\mathcal A) \to (Y,\mathcal B) $  induced by a quasi-isometry of pairs $q\colon (X,\mathcal A) \to (Y,\mathcal{B})$ is not unique. However any two such maps are at uniform $L_\infty$-distance by the first item of the previous proposition.  We will refer to any of those $\dot q$ as an \emph{induced quasi-isometry $\dot q$}.  
\end{remark}

For a group $G$ and  a collection $\calp$ of  subgroups of $G$, let $G/\calp$ denote the set of all cosets $gP$ with $g\in G$ and $P\in\calp$. A \emph{group pair} $(G,\mathcal P)$ is a pair consisting of a finitely generated group $G$  and  a finite collection of {\bf infinite}  subgroups $\mathcal P$. The collection $\calp$ is called the \emph{peripheral structure} of the pair $(G,\calp)$. 
 When considering a group pair $(G,\calp)$ we assume that $G$ is endowed with a chosen  word metric $\dist_G$, while the Hausdorff distance between subsets of $G$ is denoted as  $\Hdist_G$.  
 
  \begin{definition}[Quasi-isometry of group pairs]\label{def:qipairs2}
Consider two group pairs $(G, \mathcal{P})$ and $(H, \mathcal{Q})$. 
An \emph{$(L,C,M)$-quasi-isometry of group pairs}  $q\colon (G, \mathcal{P})\to (H, \mathcal{Q})$ is a  $(L,C, M)$-quasi-isometry of pairs $q\colon (G, G/\calp) \to (H, H/\calq)$ in the sense of Defintion~\ref{defn:quasi-isometry-pairs}.

Two group pairs $(G,\calp)$ and $(H,\calq)$ are called \emph{quasi-isometric} if there is  a quasi-isometry of pairs between them. This is an  equivalence relation on the class of group pairs. 
\end{definition}

Our convention is that any finitely generated group $G$ is identified with the group pair $(G,\emptyset)$. Under this convention,   the notion of quasi-isometric group pairs extends the notion of quasi-isometric finitely generated groups.

We also consider the stronger notion of virtually isomorphic pairs.

\begin{definition}[Virtually isomorphic pairs]\label{def:VirtIsomPair}
Two group pairs are $(G,\calp)$ and $(H,\calq)$ are \emph{virtually isomorphic} if one can transform $(G,\calp)$ to $(H,\calq)$  by a finite sequence of the following three operations and their inverses:
\begin{enumerate}

\item (Replace subgroups with commensurable subgroups) 
Replace $\calp$ with a finite collection $\calq$ such that  such that each   $Q\in \calq$ is commensurable to a conjugate of some $P\in \calp$ and viceversa.  

\item (Take a quotient by a finite normal subgroup) Replace the pair  $(G,\calp)$ with    $\left(G/N, \{PN/N \mid P\in\calp\}\right)$, where  $N$ is a finite normal subgroup of $G$.  

\item (Pass to a finite index subgroup) 
Replace $(G,\calp)$ with $(H,\calq)$, where  $H$ is a finite index subgroup of $G$ and $\calq$ is defined as follows. If  
$\{g_iP_i \colon i\in I\}$  is a collection of representatives of the orbits of the $H$-action on $G/\mathcal{P}$ by left multiplication, then $\mathcal{Q}=\{ g_iP_ig_i^{-1}\cap H \mid i\in I \}$. 
\end{enumerate}
\end{definition}

By definition, virtual  isomorphism  of pairs is an equivalence relation in the class of group pairs. 
In the first item of the definition, the condition on $\calq$ is equivalent to  the requirement that the identity map on $G$ is a quasi-isometry of pairs     $(G,\calp) \to (G,\calq)$; see Lemma~\ref{lem:commensurabilityHdist}.
In the second item, it is straightforward to  verify that the quotient map $G\to G/N$ is a quasi-isometry of pairs $(G,\calp) \to (G/N, \{PN/N\mid P\in\calp\})$. In the third item, the inclusion $H\hookrightarrow G$ is a quasi-isometry of pairs $(H,\calq) \to (G,\calp)$; a verification of this fact  can be found in~\cite[Proposition 2.13]{EMP21-qichar}.  Therefore, we have:

\begin{proposition}
 Virtually isomorphic pairs are quasi-isometric pairs.   
\end{proposition}

The notions of quasi-isometries and virtual isomorphisms of group pairs leads to the notion of a geometric property of a group pair.

\begin{definition}[Geometric properties]
A  property $\mathbb{P}$ of a group pair is \emph{geometric} if for every group pair $(G,\calp)$ which is quasi-isometric to some pair $(H,\calq)$ satisfying  $\mathbb{P}$, there exists a pair $(K, \mathcal{L})$ satisfying $\mathbb{P}$ that is virtually isomorphic to $(G,\calp)$. 
\end{definition}

Most (but not all) of the properties we consider in this paper are actually quasi-isometry invariants.  Note that if a property is a quasi-isometry invariant of group pairs, then it is geometric, but the converse does not necessarily hold; see, for example, Theorem~\ref{thm:Malnormality}.

\section{The coset space}\label{sec:CosetSpace}




In this section, we introduce the \textit{coset space} and use it to give some first examples of geometric properties of group pairs.
\begin{definition}[The coset space]
 Let $(G,\calp)$ be a group pair. 
 The \emph{coset space} is the extended metric space $(G/\calp, \Hdist_G)$. 
 The action of $G$ on $(G/\calp, \Hdist_G)$ by left multiplication is by isometries, and since $\calp$ is finite, there are only finitely  many distinct  $G$-orbits. 
\end{definition}

\begin{example} Let $A$ and $B$ be infinite finitely generated groups.
 \begin{enumerate}
     \item In the coset space of $(A\ast B, A)$, any two distinct points are at infinite distance.

     \item The coset space of $(A\times B, A)$ is quasi-isometric to $B$.

     \item The coset space of $(\ZZ^2\ast_\ZZ \ZZ^2, P)$, where $P$ is the edge group of the splitting, is quasi-isometric to  the Cayley graph of the free group of rank two.
 \end{enumerate}
\end{example}

\begin{lemma}[The coset space is uniformly locally finite]\label{lem:CosetSpaceUniLocFin}
For every   $r>0$ there is a uniform bound on the cardinality of balls of radius $r$ in $(G/\calp, \Hdist_G)$.  
\end{lemma}
\begin{proof} Since $G$ is finitely generated,  $(G/\calp, \Hdist_G)$ is locally finite. Since $\calp$ is finite, there are only finitely many $G$-orbits of points on $G/\calp$ and the statement follows.  
\end{proof}

Proposition~\ref{propF:InducedDot-q-03} and Lemma~\ref{lem:CosetSpaceUniLocFin} immediately imply that the coset space is  a quasi-isometry invariant of group pairs: 

\begin{proposition}\label{prop:QIofcosetspaces}
If $q\colon (G,\calp)\to (H,\calq)$ is a quasi-isometry of group pairs, then any induced  $\dot q\colon (G/\calp,\Hdist_G) \to (H/\calq, \Hdist_H)$ is a uniformly finite-to-one quasi-isometry of extended metric spaces. 
\end{proposition}

\subsection{Bounded packing}
Our first application of the coset space is that bounded packing, originally defined in \cite{HW09}, is a geometric property of group pairs, proving Theorem~\ref{thm:GeomProps}\eqref{item:boundedpacking}.  
\begin{definition}[Bounded packing]
A group pair $(G,\calp)$ has  \emph{bounded packing} if, for each constant $D$,
there is a number $N$ such that for any collection of $N$ distinct cosets $g_1P_1,\ldots , g_NP_N$ in $G/\calp$, at least two are separated by a distance of at least $D$ in $(G,\dist_G)$. 
\end{definition}

\begin{theorem}
Bounded packing is a quasi-isometry invariant of group pairs.    
\end{theorem}

\begin{proof}
Suppose that $q\colon (G,\calp) \to (H,\calq)$ is an $(L,C,M)$-quasi-isometry of pairs and $(H,\calq)$ has bounded packing. Consider an induced function $\dot q\colon G/\calp \to H/\calq$.
By Proposition~\ref{propF:InducedDot-q-03}, there is $K>0$ such that  $\dot q^{-1}(hQ)$ has cardinality at most $K$ for any $hQ\in H/\calq$.  

Let $D>0$. Since  $(H,\calq)$ has bounded packing, there exists $N'>0$ such that if $\mathcal{S}'\subset H/\calq$ has the property that any two left cosets in $\mathcal{S}'$ are at infimum distance at most $D'=LD+C+2M$ in $(H,\dist_H)$ then $\mathcal{S}'$ has cardinality less than $N'$. Let $N=N'\cdot K$. 

Suppose that $\mathcal{S} \subset G/\calp$ has the property that   any two left cosets in $\mathcal{S}$ are at distance at most $D$. Let $\mathcal{S}'=\{\dot q(gP)\mid gP\in \mathcal{S}\}$. Then any two left cosets of $\mathcal{S}'$ are at distance at most $D'$ and hence $\mathcal{S'}$ has cardinality at most $N'$. Since $\dot q$ is a $K$-to-one function, it follows that $\mathcal{S} \subset \dot q^{-1}(\mathcal{S}')$ has at most $N$ elements. Therefore $(G,\calp)$ has bounded packing.
\end{proof}

\subsection{Reducible pairs}
Our second application of the coset space is a geometric proof that being \textit{reducible} is invariant under quasi-isometry of group pairs. A different argument proving this result is implicit in~\cite{EMP23-QI-Dehn}; see Proposition~6.3 and Lemmas~6.9 and 6.11 in that paper.  

We begin with a discussion of the commensurator of a subgroup.  The \emph{commensurator} $\comm_G(P)$ of a subgroup $P$ of a group  $G$ is defined as
\[ \comm_G(P)=\{g\in G\mid P\cap gPg^{-1} \text{ has finite index in $P$ and in $gPg^{-1}$}\}.\] 
Observe that $\comm_G(P)$ is a subgroup of $G$. 
Commensurability of  subgroups has the following well-known  geometric interpretation; for a proof, see, for example,~\cite[Lemma 5.14]{EMP23-QI-Dehn}.

\begin{lemma}\label{lem:commensurabilityHdist}\label{prop:GeomCommensurator3}
Let $G$ be a finitely generated group with a word metric $\dist_G$, let $P$ and $Q$
be subgroups, and let $g\in G$. Then $P$ and $gQg^{-1}$ are commensurable subgroups if and only
if $\Hdist_G(P,gQ)<\infty$.   In particular, 
\[ \comm_G(P) = \{g\in G \mid \Hdist_G(P, gP)<\infty\}.\]
\end{lemma}


 


\begin{definition}[Reducible pair]
A group pair $(G,\calp)$ is \emph{reducible} if 
every $P\in\calp$ has finite index in its commensurator $\comm_G(P)$.    
\end{definition}


The following theorem proves Theorem~\ref{thm:GeomProps}\eqref{item:reducible}.

\begin{theorem}\label{thm:reducible2} 
Being reducible is quasi-isometry invariant of group pairs.    
\end{theorem}

Our  argument proceeds by showing that being reducible can be characterized geometrically in terms of the coset space.  We use the notion of bounded galaxies.


\begin{definition}[Bounded galaxies]
An extended metric space $X$ is said to have  \emph{bounded galaxies} if for every $x\in X$ the subspace $B(x,\infty)=\{y\in X\mid \dist_X(x,y)<\infty\}$ is bounded. We call $B(x,\infty)$ the \emph{galaxy containing $x$}. Note that if $B(x)\cap B(y)$ is non-empty then $B(x)=B(y)$.  
\end{definition}

\begin{lemma}\label{lem:BoundedGalaxiesQI}
 Having bounded galaxies is a quasi-isometry invariant of extended metric spaces.   
\end{lemma}
\begin{proof}
Let $f\colon X \to Y$ be a quasi-isometry of extended metric spaces and suppose that $Y$ has bounded galaxies. Let $x\in X$ and observe that $f(B(x,\infty))\subset B(f(x),\infty)$. Since $Y$ has bounded galaxies, $B(f(x),\infty)$ is a bounded subset of $Y$. Since $f$ is a quasi-isometry, it follows that $B(x,\infty)$ is bounded.  
\end{proof}

Using Lemma~\ref{lem:commensurabilityHdist}, we now give a geometric characterization of reducibility.
\begin{lemma}\label{prop:ReducibleCharacterization}     
Let $(G,\calp)$ be a group pair. The following statements are equivalent.
\begin{enumerate}
    \item $(G,\calp)$ is reducible.
    \item For every  $Q\in G/\calp$, the set $\{R\in G/\calp \mid \Hdist_G(R,Q)<\infty\}$ is bounded.
    \item There is constant $N>0$ such that for every  $Q\in G/\calp$, the set $\{R\in G/\calp \mid \Hdist_G(R,Q)<\infty\}$ has cardinality at most $N$.  
\end{enumerate}
\end{lemma}

\begin{proof}
Since $(G/\calp,\Hdist_G)$ is locally finite, a subset of $G/\calp$ is bounded if and only if it is finite. Hence all the galaxies of $(G/\calp,\Hdist_G)$ are bounded if and only if all  galaxies are finite.
Since the $G$-action on  $(G/\calp, \Hdist_G)$  has only finitely distinct $G$-orbits of points, there are finitely orbits of galaxies. Therefore all galaxies of $(G/\calp,\Hdist_G)$ are bounded if and only if there is $N>0$ such that every galaxy has at most $N$ points. This proves the equivalence of (2) and (3).

Now we argue that (1) and (2) are equivalent. The coset space $(G/\calp,  \Hdist_G)$ has a natural partition into a finite number of subspaces, namely, $G/\calp = \bigcup_{P\in\calp} G/P$. The intersection of a galaxy of $G/\calp$ with a subspace $G/P$ is either a galaxy of $G/P$ or is empty. Hence $(G/\calp,\Hdist_G)$ has bounded galaxies if and only if $(G/P, \Hdist_G)$ has bounded galaxies for every $P\in\calp$.  To conclude that (1) and (2) are equivalent,  observe that
Lemma~\ref{prop:GeomCommensurator3} implies that 
$P$ has finite index in $\comm_G(P)$ if and only if $(G/P, \Hdist_G)$ has finite (and hence bounded) galaxies . 
\end{proof}

We are now ready to prove Theorem~\ref{thm:reducible2}.

\begin{proof}[Proof of Theorem~\ref{thm:reducible2}]

Since quasi-isometric group pairs have quasi-isometric coset spaces by Proposition~\ref{prop:QIofcosetspaces}, it suffices to prove that a  group pair is reducible if and only if its coset space has bounded galaxies.   This follows immediately from Lemma~\ref{prop:ReducibleCharacterization}.
\end{proof}



\subsection{Reduced pairs}
In this section, we show that being a reduced pair is a geometric property of group pairs; see Theorem~\ref{thm:reducible}.

\begin{definition}[Reduced pair] \label{def:reducedpair}
A group pair $(G,\calp)$ is \emph{reduced} if    $P=\comm_G(P)$  for each $P\in \calp$ and no pair of subgroups in $\calp$ are conjugate.  
\end{definition}

Notice that a group pair $(G,\calp)$ is reduced if for any $P,Q\in\calp$ and $g\in G$, if $P$ and $gQg^{-1}$ are commensurable subgroups then $P=Q$ and $g\in P$.

\begin{example}\label{ex:reducednotQIinvar}
Let $A$ and $B$ be infinite cyclic groups, and let $2A$ be the index two subgroup of $A$. Then the pairs $(A\ast B, A)$ and $(A\ast B, 2A)$ are quasi-isometric, $(A\ast B, A)$ is a reduced pair, and $(A\ast B, 2A)$ is reducible but is not reduced.  
\end{example}

\begin{lemma}\label{lem:HdistComm}
Let $P$ and $Q$ be subgroups of a finitely generated group $G$. If $\Hdist_G(P,Q)<\infty$ then $\comm_G(P)=\comm_G(Q)$.    
\end{lemma}
\begin{proof}
Suppose $\Hdist_G(P,Q)<\infty$. Since $\Hdist_G$ is invariant under multiplication on the left, the triangle inequality in $(G/\calp, \Hdist_G)$ implies that 
$ \Hdist_G(P,gP)<\infty $ if and only if $ \Hdist_G(Q,gQ)<\infty $  for every $g\in G$.  The conclusion then follows directly from  Lemma~\ref{lem:commensurabilityHdist}.
\end{proof}

Reduced pairs can be characterized geometrically as follows.  We denote by $A^g$ the conjugate $gAg^{-1}$.

\begin{proposition}[Reduced pair  characterization]\label{prop:RefinedChar3}
A group pair is reduced if and only if the distance in the coset space only take the values zero and infinity.    
\end{proposition}
\begin{proof}
Suppose $(G,\calp)$  is reduced.  Let $g_1P_1, g_2P_2\in G/\calp$  satisfy $\Hdist_G(g_1P_1, g_2P_2)$ is finite. Lemma~\ref{lem:HdistComm} implies that $\comm_G(P_1^{g_1})= \comm_G(P_2^{g_2})$. Since   
  $(G,\calp)$ is  reduced, we have that $P_1^{g_1}=\comm_G(P_1^{g_1})$ and $ \comm_G(P_2^{g_2})=P_2^{g_2}$, and so $g_1P_1g_1^{-1}=g_2P_2g_2^{-1}$. Since 
$(G,\calp)$ is reduced, no pair of subgroups in $\calp$ are conjugate, and hence $P_1=P_2$ and $g_2^{-1}g_1\in \comm_G(P_1)=P_1$. Therefore  $g_1P_1=g_2P_2$ and $\Hdist_G(g_1P_1,g_2P_2)=0$.  

Conversely, suppose that $\Hdist_G$ only takes the values zero and infinity on $G/\calp$.  By Lemma~\ref{prop:GeomCommensurator3}, it follows that $P=\comm_G(P)$ for every $P\in \calp$. If $P,Q\in\calp$ and $P=gQg^{-1}$ then $\Hdist_G(P, gQ)<\infty$ and hence $P=gQ$.  Hence $P=Q$ and $g\in P$, and so $(G,\calp)$ is reduced.
\end{proof}

To prove that being a reduced pair is a geometric property, we need the notion of a refinement.

 \begin{definition}[Refinement  $(G,\calp^*)$]
  A \emph{refinement} $\calp^*$ of a collection of subgroups $\calp$ of $G$ is a set of representatives of conjugacy classes of the collection of subgroups $\{\comm_G(gPg^{-1}) \colon P\in\calp \text{ and } g\in G \}$. 
 \end{definition}

 It is straightforward from the definition of virtually isomorphic group pairs (see Definition~\ref{def:VirtIsomPair}) to see that if $(G,\calp)$ is reducible, then it is virtually isomorphic to its refinement $(G,\calp^*)$; see \cite[Proposition 6.3]{HMS2021}.

\begin{remark}[Refinements are reduced]\label{rem:RefinementsRedued2}
Since $\comm_G(\comm_G(P))=\comm_G(P)$ and $\comm_G(gPg^{-1})=g\comm_G(P)g^{-1}$, it follows that any refinement $(G,\calp^*)$ is reduced.    
\end{remark}


We now give an alternative proof of Theorem~\ref{thm:GeomProps}\eqref{item:reduced}, originally proven in \cite{EMP23-QI-Dehn}.


\begin{theorem}\label{thm:reducible}
Being reduced is a geometric property of group pairs.
\end{theorem}

\begin{proof} First, observe that a reduced pair is reducible. Since being reducible is geometric by Theorem~\ref{thm:reducible2},  a group pair quasi-isometric to a reduced pair is reducible.  Finally, every reducible pair is virtually isomorphic to its refinement.
\end{proof}

Note that the virtual isomorphism in the proof is necessary.  In particular, the property of being reduced is \textit{not} a quasi-isometry invariant of group pairs; see Example~\ref{ex:reducednotQIinvar}.


\subsection{Almost Malnormality}

We now show that almost malnormality is a geometric property of group pairs, proving Theorem~\ref{thm:GeomProps}\eqref{item:malnormal}.



\begin{definition}[Almost malnormal]\label{def:malnormal}
A collection of subgroups $\calq$ of a group $G$ is \emph{almost malnormal} if for any $Q,Q'\in\calp$ and $g\in G$, either $gQg^{-1} \cap Q'$ is finite or $Q=Q'$ and $g\in Q$. A group pair $(G,\calp)$ is called \emph{almost malnormal} if $\calp$ is an almost malnormal collection in $G$.
\end{definition}

\begin{theorem}\label{thm:Malnormality}
Being almost malnormal is a geometric property of group pairs.    
\end{theorem}





\begin{proof}
Suppose that $(G,\calp)$ and $(H,\calq)$ are quasi-isometric group pairs and $(H,\calq)$ is almost malnormal. 
Let $(G,\calp^*)$ be a refinement of $(G,\calp)$.
By \cite[Theorem D]{EMP23-QI-Dehn}, we have that $(G,\calp^*)$ is almost malnormal.  It remains to show that $(G,\calp)$ is virtually isomorphic to $(G,\calp^*)$.  For this, note that $(H,\calq)$ is reducible, as it is almost malnormal.  Thus Theorem~\ref{thm:reducible2} implies that $(G,\calp)$ is reducible, as well, and so $(G,\calp)$ is virtually isomorphic to $(G,\calp^*)$.   
\end{proof}

\section{The coset intersection complex}\label{sec:CosetIntCx}

In this section, we introduce our primary object of study, the \textit{coset intersection complex} $\mc K(G,\calp)$ associated to a group pair $(G,\calp)$.  In the first two subsections, we study geometric properties and the homotopy type of $\mc K(G,\calp)$, respectively. In the third subsection, we study the behavior of $\mc K(G,\calp)$ under quasi-isometries of pairs.

\begin{definition}[Coset intersection complex $\mathcal{K}(G,\calp)$] Let $(G,\calp)$ be a group pair. The \emph{coset intersection complex} is the  simplicial complex with vertex set $G/\calp$ and such that a set $\{g_1P_1,\ldots ,g_kP_k\}\subset G/\calp$ defines a  simplex if $\bigcap_{i=1}^k g_iP_ig_i^{-1}$ is an infinite subgroup of $G$.
\end{definition}

 Observe that the $G$-action on $G/\calp$ extends to a cellular $G$-action on $\mathcal{K}(G,\calp)$ with finitely many $G$-orbits of $0$-cells; however the action is  not necessarily cocompact.  
 
\begin{example}\label{ex:BasicExs}
\begin{enumerate}
 \item The group pair $(G, \calp)$ is   almost malnormal if and only if the coset intersection complex $\mathcal{K}(G,\calp)$ is zero dimensional.  This is Theorem~\ref{thm:alg&geom}\eqref{item:0diml}.

\item The coset intersection complex $\mathcal{K}(\ZZ^2,\ZZ\times\{0\})$ is the flag complex with $1$-skeleton the complete graph with vertex set $\ZZ$.  In particular, $\mc K(\ZZ^2,\ZZ\times\{0\})$ is an infinite-dimensional simplex, and the $\ZZ^2$-action on its $1$-skeleton is not cocompact.

\item Consider the free group $G$ with free generating set $\{x_0,x_1,x_2\}$ and the collection of subgroups $\calp $ consisting of  $P_i=\langle x_{i}, x_{(i+1) \mod 3} \rangle$ for $i=0,1,2$. In this case, $\mathcal{K}(G,\calp)$ is cocompact connected $G$-graph. The figure below illustrates a part of this graph: 
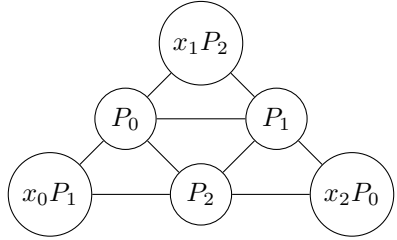
\begin{equation}\nonumber
  \begin{tikzpicture}
	\begin{pgfonlayer}{nodelayer}
		\node [style=circle] (7) at (-2, 0) {$P_1$};
		\node [style=circle] (8) at (-4, -2) {$P_2$};
		\node [style=circle] (9) at (-6, 0) {$P_0$};
		\node [style=circle] (12) at (-8, -2) {$x_0P_1$};
		\node [style=circle] (14) at (0, -2) {$x_2P_0$};
		\node [style=circle] (16) at (-4, 2) {$x_1P_2$};
	\end{pgfonlayer}
	\begin{pgfonlayer}{edgelayer}
		\draw (8) to (7);
		\draw (8) to (9);
		\draw (9) to (7);
		\draw (12) to (8);
		\draw (12) to (9);
		\draw (9) to (16);
		\draw (16) to (7);
		\draw (7) to (14);
		\draw (14) to (8);
	\end{pgfonlayer}
\end{tikzpicture}

\end{equation}

 \item If  $Q$ is a quasi-convex subgroup of a hyperbolic group $G$, then $\mathcal{K}(G,Q)$ is a $G$-cocompact simplicial complex.  In particular $\mathcal{K}(G,Q)$ is finite dimensional. This follows from an argument of~\cite{GMJRS98}; see Example~\ref{ex:HypQc} for details.

\item If $G=BS(1,k)=\langle a,t\mid tat^{-1}=a^k \rangle$   is a solvable Baumslag-Solitar group, then  $\mathcal{K}(G,  \langle t \rangle)$ is an infinite edgeless graph.

\item Consider the group $G=F_2\times \mathbb{Z}=\langle x,y,t\mid [x,t],\ [y,t] \rangle$ and the collection of subgroups $\mathcal{P}=\{ \langle x,t\rangle,\  \langle y,t\rangle \}$.
Since any conjugate of a subgroup in $\calp$ contains the subgroup $\langle t \rangle$, the coset intersection complex $\mathcal{K}(G,\calp)$ is the flag complex with $1$-skeleton the complete graph
with vertex set $G/\calp$. In particular, $\mathcal{K}(G,\calp)$ is an infinite-dimensional simplex. 

\item The coset intersection complex $\mathcal{K}$ of a right-angled Artin group $A(\Gamma)$ with respect to   maximal standard abelian subgroups is infinite dimensional if the defining graph $\Gamma$ contains a $2$-path as the previous example illustrates. However, in general, $\mathcal{K}$ is not an infinite dimensional simplex. Consider 
 $F_2\times F_2=\langle x,y,s,t\mid [x,s], [x,t],[y,s],[y,t] \rangle$ and the collection of subgroups $\mathcal{P}=\{ \langle x,s \rangle,\ \langle x,t \rangle,\ \langle y,s \rangle,\  \langle y,t\rangle\}$. Then $\mathcal{K}(G,\calp)$ contains the $4$-cycle in the figure
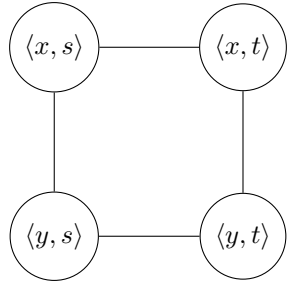
\begin{equation}\nonumber
  \begin{tikzpicture}
	\begin{pgfonlayer}{nodelayer}
		\node [style=circle] (4) at (-6, 3) {$\langle x,s\rangle$};
		\node [style=circle] (6) at (-1, 3) {$\langle x,t\rangle$};
		\node [style=circle] (7) at (-1, -2) {$\langle y,t\rangle$};
		\node [style=circle] (8) at (-6, -2) {$\langle y,s\rangle$};
	\end{pgfonlayer}
	\begin{pgfonlayer}{edgelayer}
		\draw (4) to (8);
		\draw (8) to (7);
		\draw (7) to (6);
		\draw (6) to (4);
	\end{pgfonlayer}
\end{tikzpicture}

\end{equation}
and the diagonals are not part of the complex. We expect that if $\Gamma$ is not simply-connected, then neither is $\mathcal K$.  


\end{enumerate}
\end{example}

\subsection{Simplicial maps between coset intersection complexes}
 In this section, we  illustrate  the geometric nature of the coset intersection complex  $\mathcal{K}(G,\calp)$. We continue to use the notation of Proposition~\ref{propF:InducedDot-q-03}, where it was shown that  a quasi-isometry $q$ of pairs induces a quasi-isometry $\dot q$ of coset spaces. Recall that  a \emph{simplicial map $f\colon X \to Y$} is a  function between the sets of vertices of the  simplicial complexes $X$ and $Y$ with the property that the images of the vertices of a simplex always span a simplex. 

     

\begin{proposition}\label{prop:SimplicalCosetIntersection02}
Let $q\colon (G,\calp) \to (H,\calq)$ be a quasi-isometry of pairs. Then any induced quasi-isometry $\dot q\colon (G/\calp, \Hdist_G) \to (H/\calq, \Hdist_H)$ is a simplicial map $\dot q\colon \mathcal{K}(G,\calp) \to \mathcal{K}(H,\calq)$. Moreover, $\dot q\colon \mathcal{K}(G,\calp)^{(1)} \to \mathcal{K}(H,\calq)^{(1)}$  is also a quasi-isometry. 
\end{proposition}

The proof of the proposition relies on some well-know facts summarized in the lemmas below. The first  lemma has appeared in a number of articles, see,  for example,~\cite[Lemma 2.2]{MSW11}, \cite[Lemma 4.7]{MP07}, ~\cite[Proposition 9.41]{HK10} and~\cite[Corollary 4.6]{HW09}. For a subset $A \subset G$, let $\mathcal{N}_r(A)$ denote the closed $r$-neighborhood of $A$ in $(G,\dist_G)$.  

\begin{lemma}[Fundamental lemma] \label{lem:fundamental}
 Let $G$ be a finitely generated group with word metric $\dist_G$, and let  $g_1P_1, \ldots, g_kP_k$ be arbitrary left cosets of subgroups of $G$.  For any $r>0$, there exists $R>0$ such that   \[  \bigcap_{i=1}^k \mathcal{N}_r(g_iP_i) \subseteq \mathcal{N}_R\left( \bigcap_{i=1}^k g_iP_ig_i^{-1} \right) .\]   
 \end{lemma}

The fundamental lemma provides the following geometric interpretation of  simplices of the coset intersection complex. 
 
\begin{lemma}[$\mathcal{K}(G,\calp)$ is geometric] \label{prop:GeometryCosetIntersectionComplex2}
 For a group pair $(G,\calp)$ and any finite subset $\{g_1P_1, g_2P_2, \dots , g_kP_k\}$ of $G/\calp$,  the following statements are equivalent.
 \begin{enumerate}
     \item The  group $\bigcap_{i=1}^kg_iP_ig_i^{-1}$ is  infinite.
     \item The set $\bigcap_{i=1}^k\mathcal{N}_r(g_iP_i)$ is   unbounded in $(G,\dist_G)$ for some $r>0$.
 \end{enumerate} 
\end{lemma}
\begin{proof}
Observe that the subgroup $\bigcap_{i=1}^kg_iP_ig_i^{-1}$ is a subset of  $\bigcap_{i=1}^k\mathcal{N}_r(g_iP_i)$   for $r>\max_i\{ \dist_G(e,g_i)\}$.  Since $(G,\dist_G)$ is a locally finite metric space, if $\bigcap_{i=1}^kg_iP_ig_i^{-1}$ is infinite then is  unbounded, and hence $\bigcap_{i=1}^k\mathcal{N}_r(g_iP_i)$ is unbounded as well.

 Conversely, suppose $\bigcap_{i=1}^k\mathcal{N}_r(g_iP_i)$ is unbounded. By Lemma~\ref{lem:fundamental}, there exists $R>0$ such that  
$\mathcal{N}_R\left(\bigcap_{i=1}^kg_iP_ig_i^{-1}\right)$ is unbounded. Since $(G,\dist_G)$ is locally finite, it follows that  $\bigcap_{i=1}^kg_iP_ig_i^{-1}$ is an infinite subgroup.
\end{proof}

The next lemma relates intersections of conjugates of a subgroup and the corresponding intersections in the image under an induced quasi-isometry $\dot q$.
\begin{lemma}\label{lem:simplices}
 Let $q\colon (G,\calp) \to (H,\calq)$ be a quasi-isometry of pairs, and let $\dot q$ be an induced quasi-isometry of coset spaces. 
Let $g_iP_i \in G/\calp$  and   $h_iQ_i = \dot q(g_iP_i)$ for $i=1,\ldots ,k$. Then   $\bigcap_{i=1}^k g_iP_ig_i^{-1}$ is infinite  
if and only if 
$\bigcap_{i=1}^k h_iQ_ih_i^{-1}$ is  infinite.
\end{lemma}
\begin{proof}
Since $q\colon (G,\dist_G) \to (H,\dist_H)$ is a quasi-isometry, the set 
$\bigcap_{i=1}^k \mathcal{N}_r(g_iP_i)$ is unbounded for some $r$ if and only if 
$\bigcap_{i=1}^k  \mathcal{N}_{r'}(  q(g_iP_i) )$
is unbounded for some $r'$. 
Since $q\colon (G,\calp) \to (H,\calq)$ is a quasi-isometry of pairs,  the Hausdorff distance $\Hdist_H(  h_iQ_i, q(g_iP_i))$ is uniformly  bounded by  Proposition~\ref{propF:InducedDot-q-03}. Hence $\bigcap_{i=1}^k \mathcal{N}_r(g_iP_i)$ is unbounded for some $r$ if and only if 
$\bigcap_{i=1}^k  \mathcal{N}_{r'} (  h_iQ_i)  $
is unbounded for some $r'$. Therefore, it follows from  Lemma~\ref{prop:GeometryCosetIntersectionComplex2} that   
$\bigcap_{i=1}^k g_iP_ig_i^{-1}$ is infinite  
if and only if 
$\bigcap_{i=1}^k h_iQ_ih_i^{-1}$ is infinite.  
\end{proof}

The next lemma shows that when the induced map $\dot q$ is a simplicial map, it restricts to a quasi-isometry on the 1--skeleton.
\begin{lemma}\label{lem:DotqIsQI}
 Let $q\colon (G,\calp) \to (H,\calq)$ be a quasi-isometry of pairs with quasi-inverse $p\colon (H,\calq) \to (G,\calp)$.
If $\dot q\colon \mathcal{K}(G,\calp) \to \mathcal{K}(H,\calq)$ and $\dot p\colon \mathcal{K}(H,\calq) \to \mathcal{K}(G,\calp)$ are  simplicial maps, then they restrict to  quasi-isometries of the 1--skeletons of the complexes.
\end{lemma}
\begin{proof}
Since $\dot q$ and $\dot p$ are   simplicial maps, they are also   Lipschitz maps with respect to the combinatorial metric on 1-skeleta.  
Let $r=p\circ q$. We will show that the  simplicial map 
$\dot r  \colon \mathcal{K}(G,\calp) \to \mathcal{K}(G,\calp)$
is at bounded distance from the identity map  $\dot{\mathbb{I}}$  on $\mathcal{K}(G,\calp)$ and hence, by symmetry, $\dot p$ and $\dot q$ are quasi-isometries.  

Since $p$ is a quasi-inverse of $q$, there is a constant $D$ such that $r=p\circ q$ 
is at 
distance at most $D$ from the identity map on $G$, and hence  
\begin{equation}\nonumber
\begin{split}
\Hdist_G(\dot r(gP),\dot{\mathbb{I}}(gP))& = \Hdist_G(\dot r(gP), gP)\leq D<\infty .
\end{split}
\end{equation}
In particular, $\dot r \colon \mathcal{K}(G,\calp) \to \mathcal{K}(G,\calp)$ satisfies
that for every vertex $gP$ of $\mathcal{K}(G,\calp)$, either $gP$ and $\dot r(gP)$ are the same vertex, or $\{gP, \dot r(gP)\}$ is a 1-cell of $\mathcal{K}(G,\calp)$. This implies that the maps $\dot r$ and $\dot{\mathbb{I}}$ are at distance at most one.   
\end{proof}

We are now ready to prove Proposition~\ref{prop:SimplicalCosetIntersection02}.  
\begin{proof}[Proof of Proposition~\ref{prop:SimplicalCosetIntersection02}]
Let $g_iP_i \in G/\calp$  and   $h_iQ_i = \dot q(g_iP_i)$ for $i=1,\ldots ,k$. Suppose that $\{g_1P_1,\ldots, g_kP_k\}$ is a simplex of $\mathcal{K}(G,\calp)$. By definition,  the intersection $\bigcap_{i=1}^k g_iP_ig_i^{-1}$ is infinite.  Lemma~\ref{lem:simplices} implies that 
$\bigcap_{i=1}^k h_iQ_ih_i^{-1}$ is infinite, and hence $\{h_1Q_1,\ldots , h_kQ_k\}$ is a simplex of $\mathcal{K}(H,\calq)$. Therefore  $\dot q\colon \mathcal{K}(G,\calp) \to \mathcal{K}(H,\calq)$ is a simplicial map, and, by Lemma~\ref{lem:DotqIsQI}, it restricts to a a quasi-isometry  $\dot q\colon \mathcal{K}(G,\calp)^{(1)} \to \mathcal{K}(H,\calq)^{(1)}$.
\end{proof}

Considering $\mc K(G,\calp)$ as a regular piecewise Euclidean complex, that is, one in which each simplex is Euclidean and has unit edge length,  it admits a natural length metric; see \cite[Definition~I.7.4 \& 
Corollary~I.7.10]{BridsonHaefligerBook}.
\begin{corollary}\label{cor:QIofComplexes}
    Let $q\colon (G,\calp) \to (H,\calq)$ be a quasi-isometry of pairs. Any induced map $\dot q\colon \mathcal{K}(G,\calp) \to \mathcal{K}(H,\calq)$ is a quasi-isometry of regular piecewise Euclidean complexes.
\end{corollary}

\begin{proof}
    Each simplex $\Delta$ has uniformly bounded diameter, independent of the dimension of $\Delta$, and so each $\Delta$ is uniformly quasi-isometric to its 1-skeleton.  It is straightforward to verify that these quasi-isometries extend to a quasi-isometry $\mc K(G,\calp) \to \mc K(G,\calp)^{(1)}$ using the length metric above.  Combining this with Proposition~\ref{prop:SimplicalCosetIntersection02} completes the proof. 
\end{proof}

\subsection{Isomorphisms between  coset intersection complexes} Recall that a pair $(G,\calp)$ is reduced if    $P=\comm_G(P)$  for each $P\in \calp$ and no pair of subgroups in $\calp$ are conjugate; see Definition~\ref{def:reducedpair}.  In our context, the relevance of  reduced group pairs is that quasi-isometries between reduced group pairs  induce  isomorphisms of the corresponding coset intersection complexes as the following proposition shows.

 \begin{proposition}
\label{lem:QI-to-refinedPair}  \label{prop:IsomorphicCosetIntersectionCmplex2}

 Let $q\colon (G,\calp) \to (H,\calq)$ be a quasi-isometry of pairs.
 \begin{enumerate}
     \item If $(H,\calq)$ is reduced, then 
 the induced simplicial map $\dot q\colon \mathcal{K}(G,\calp) \to \mathcal{K}(H,\calq)$ is surjective, and its restriction to the zero-skeleton is uniformly finite-to-one.

 \item If $(G,\calp)$ is reduced,  then 
   $\dot q\colon \mathcal{K}(G,\calp) \to \mathcal{K}(H,\calq)$ is an embedding of simplicial complexes. 
 
 \item If $(G,\calp)$ and $(H,\calq)$ are both reduced, then 
$\dot q\colon \mathcal{K}(G,\calp) \to \mathcal{K}(H,\calq)$ is an isomorphism of simplicial complexes.

\end{enumerate}
\end{proposition}
 
 \begin{proof}
  Let $p\colon (H,\calq) \to (G,\calp)$ be a quasi-inverse of $q$. By Proposition~\ref{propF:InducedDot-q-03},  the induced maps $\dot q \colon (G/\calp,\Hdist_G) \to (H/\calq, \Hdist_H)$ and $\dot p \colon (H/\calq,\Hdist_H) \to (G/\calp, \Hdist_G)$  are quasi-isometries and  quasi-inverses.  Moreover,  $(G/\calp, \Hdist_G)$ and $(H/\calq, \Hdist_H)$ are uniformly locally finite by Lemma~\ref{lem:CosetSpaceUniLocFin}, and so $\dot p$ and $\dot q$ are uniformly finite-to-one functions. 

If   $(H,\calq)$ is reduced, then any two distinct points of $(H/\calq, \Hdist_H)$ are at infinite distance. Since $\dot q \circ \dot p \colon (H/\calq,\Hdist_H) \to (H/\calq,\Hdist_H)$ is a quasi-isometry, it follows that $\dot q \circ \dot p \colon (H/\calq,\Hdist_H) \to (H/\calq,\Hdist_H)$ has to be a bijection. In fact, $\dot q \circ \dot p$ must be the identity, since $ \Hdist_G (\dot q \circ \dot p (hQ) , hQ ) <\infty$ implies $\dot q \circ \dot p (hQ)=hQ$.  On the other hand, by Proposition~\ref{prop:SimplicalCosetIntersection02}, we have that $\dot q$ and $\dot p$ are simplicial maps $\dot q\colon \mathcal{K}(G,\calp) \to \mathcal{K}(H,\calq)$ and $\dot p\colon \mathcal{K}(H,\calq) \to \mathcal{K}(G,\calp)$. Hence  $\dot q \circ \dot p\colon \mathcal{K}(H,\calq) \to \mathcal{K}(H,\calq)$ is a simplicial map which is the identity at the level of the zero skeleton, and therefore $\dot q \circ \dot p$ is the identity map on $\mathcal{K}(H,\calq)$. This implies that $\dot q\colon \mathcal{K}(G,\calp) \to \mathcal{K}(H, \calq)$ is surjective and $\dot p\colon \mathcal{K}(H,\calq) \to \mathcal{K}(G,\calp)$ is an embedding. This completes the proof of the first two statements; the third statement is immediate.
\end{proof}

\subsection{Homotopy equivalences between coset intersection complexes}  In this section, we prove that a quasi-isometry of pairs induces a homotopy equivalence of coset intersection complexes.  This, along with 
Corollary~\ref{cor:QIofComplexes}, proves Theorem~\ref{thm:SimplicialMap}.

 \begin{theorem}\label{thm:HE}
Let $q\colon (G,\calp) \to (H,\calq)$ be a quasi-isometry of pairs.  An  induced simplicial map $\dot q \colon \mathcal{K}(G,\calp) \to \mathcal{K}(H,\calq)$ is a homotopy equivalence. 
Moreover, any two such induced simplicial maps are homotopic.
 \end{theorem}
 
 We begin with a few general lemmas about simplicial complexes.  For a finite set $X$, let $\Delta(X)$ denote the simplex with vertex set $X$. Abusing notation, we will not distinguish the abstract simplex $\Delta(X)$ and its topological realization. In particular, every point of $\Delta(X)$ is a formal finite sum $\sum_{x\in X} \alpha_x x$ where $\alpha_x\in [0,1]$ and $\sum_x \alpha_x =1$.

\begin{lemma}\label{lem:deformationRetraction02}
Let $\mathcal{L}$ be a subcomplex of a  simplicial complex   $\mathcal{K}$. Let $r\colon \mathcal{K}^{(0)} \to \mathcal{L}^{(0)}$ be a function such that
\begin{enumerate}
    \item $r(x)=x$ for all $x\in \mathcal{L}^{(0)}$, and
    \item for every simplex $\Delta(X)$ of $\mathcal{K}$,   $\Delta(X\cup r(X))$ is a simplex of $\mathcal{K}$.  
\end{enumerate}
Then $r$ extends to a deformation retraction $r\colon \mathcal{K} \times I \to \mathcal{L}$ such that its restriction to any simplex $\Delta(X)$ of $\mathcal{K}$ is given by 
\begin{equation} \label{eq:retraction01} r\left(\sum_{x\in X}\alpha_x x,\ t\right) = \sum_{x\in X\cap \mathcal{L}} \alpha_x x + \sum_{x\in X-\mathcal{L}} \alpha_x \left(  t r(x) +  (1-t)  x  \right) \end{equation}
\end{lemma}
\begin{proof}[Sketch of the proof]
 Every simplex $\Delta(X)$ of $\mathcal{K}$ is a face of the simplex $\Delta(X\cup r(X))$ of $\mathcal{K}$ by the second assumption. Hence the right-hand side of~\eqref{eq:retraction01} defines a  point of $\mathcal{K}$. Then, using the first assumption, it is an observation that~\eqref{eq:retraction01} defines a deformation retraction of $\mathcal{K}$ onto $\mathcal{L}$.  
\end{proof}

The next lemma gives a sufficient condition for two simplicial maps to be  homotopic. 
\begin{lemma}\label{lem:HomotopyPair}
Let $\mathcal{K}$ and $\mc K'$ be simplicial complexes, and let   $r\colon \mathcal{K} \to \mathcal{K}'$ and $s\colon \mathcal{K} \to \mathcal{K}'$ be simplicial maps.
Suppose that 
for any finite-dimensional simplex $\Delta(X)$ of $\mathcal{K}$, $\Delta(r(X)\cup s(X))$ is a simplex of $\mathcal{K}'$. Then  $r$ and $s$ are homotopic.  
\end{lemma}
\begin{proof}
    Define a homotopy $F\colon \mathcal{K}\times I \to \mathcal{K}'$  such that its restriction to any finite-dimensional simplex $\Delta(X)$ of $\mathcal{K}$ is given by
\[ F\left(\sum_{x\in X}\alpha_x x,\ t\right) =   \sum_{x\in X} \alpha_x \left(  t   s(x) +  (1-t) r(x)  \right). \]
The right-hand side is well-defined by the assumption that $\Delta(r(X)\cup s(X))$ is a simplex of $\mc K'$.  It is clearly continuous on every finite subcomplex, and so by the weak topology it is continuous.  Hence $F$ is  a homotopy from $r$ to $s$.
\end{proof}

 We now apply Lemma~\ref{lem:HomotopyPair} to two induced maps of coset intersection complexes.
\begin{lemma}\label{lem:HomotopyQIpairs}
 Let $p$ and $q$ be quasi-isometries of pairs $(G,\calp) \to (H,\calq)$, and let $\dot p \colon \mathcal{K}(G,\calp) \to \mathcal{K}(H,\calq)$ and  $\dot q \colon \mathcal{K}(G,\calp) \to \mathcal{K}(H,\calq)$ be the corresponding induced simplicial maps. If   $\Hdist_H(\dot p (gP), \dot q(gP) )<\infty$ for every $gP\in G/\calp$, then $\dot p$ and $\dot q$ are homotopic.    
\end{lemma}
\begin{proof}
It is enough to   verify the hypothesis of Lemma~\ref{lem:HomotopyPair} for $\dot p $ and $\dot q$. Consider a finite-dimensional simplex $\{g_1P_1,\ldots ,g_kP_k\}$   of $\mathcal{K}(G,\calp)$. We will show that $\{\dot p(g_iP_i)\mid 1\leq i\leq k\}\cup \{\dot q(g_iP_i)\mid 1\leq i\leq k\} $ is a simplex of $\mathcal{K}(H,\calq)$.
By assumption $ \Hdist_H(\dot p(g_iP_i),  \dot q(g_iP_i))$ is finite   
and so there exists $r>0$ such that $\dot p(g_iP_i) \subset \mathcal{N}_r(  \dot q(g_iP_i)) $   
for $1\leq i\leq k$.  
It follows that 
\begin{equation*}
 \bigcap_{i=1}^k \mathcal{N}_r(\dot p(g_iP_i)) \subset \bigcap_{i=1}^k \mathcal{N}_{2r}(\dot q(g_iP_i)),    
\end{equation*}
and hence
\begin{equation*}
    \bigcap_{i=1}^k \mathcal{N}_r(\dot p(g_iP_i)) \subset \bigcap_{i=1}^k \mathcal{N}_{2r}(  \dot q(g_iP_i))\cap \mathcal{N}_{2r}(\dot p(g_iP_i)).
\end{equation*}
Since $\{g_1P_1,\ldots ,g_kP_k\}$ is a simplex of $\mathcal{K}(G,\calp)$ and $\dot p$ is simplicial, we have that $\{\dot p(g_1P_1),\ldots ,\dot p(g_kP_k)\}$ is a simplex of $\mathcal{K}(H,\calq)$. By increasing $r$ if necessary,   Lemma~\ref{prop:GeometryCosetIntersectionComplex2} implies that $\bigcap_{i=1}^k \mathcal{N}_r(\dot p(g_iP_i))$ is unbounded in $(H,\dist_H)$, and therefore  $\bigcap_{i=1}^k \mathcal{N}_{2r}( \dot q(g_iP_i))\cap \mathcal{N}_{2r}(\dot p(g_iP_i))$ is unbounded in $(G,\dist)$. By Lemma~\ref{prop:GeometryCosetIntersectionComplex2},
$\{\dot p(g_iP_i)\mid 1\leq i\leq k\}\cup \{\dot q(g_iP_i)\mid 1\leq i\leq k\} $ is therefore a simplex of $\mathcal{K}(H,\calq)$, which completes the proof.
\end{proof}

\begin{example}
 Let $q\colon (G,\calp) \to (G,\calp)$ be a self-quasi-isometry of pairs. Clearly, the identity map $(G,\calp) \to (G,\calp)$ is also a self-quasi-isometry of pairs.  However, the induced simplicial map $\dot q \colon \mathcal{K}(G,\calp) \to \mathcal{K}(G,\calp)$  is not necessarily homotopic to the identity map on $\mathcal{K}(G,\calp)$. Consider the case of a  free group of rank two $G=\langle x,y \rangle$, where $\mathcal{P}=\{ \langle x\rangle\}$ and $q\colon G \to G$ is given by $q(g)=yg$. In this case $q$ is a self-quasi-isometry of pairs $q\colon (G,\calp) \to (G,\calp)$; in fact, it is an isometry of $G$. 
 The complex $\mathcal{K}(G,\calp)$ is an infinite discrete set, and so a map homotopic to the identity must be the identity. 
 The induced $\dot q\colon \mathcal{K}(G,\calp) \to \mathcal{K}(G,\calp)$, however, is given by $\dot q (g\langle x\rangle)=yg\langle x\rangle$,  which is not the identity. 
\end{example}



\begin{proposition}\label{prop:HomotopyComposition}
Let $q\colon (G,\calp) \to (H,\calq)$ and $p\colon (H,\calq) \to (K,\calr)$  be  quasi-isometries of pairs and let $r=p\circ q$. 
Let $\dot q \colon \mathcal{K}(G,\calp) \to \mathcal{K}(H,\calq)$,  $\dot p \colon \mathcal{K}(H,\calq) \to \mathcal{K}(K,\calr)$ and $\dot r\colon \mathcal{K}(G,\calp) \to \mathcal{K}(K,\calr)$ be corresponding induced simplicial maps. Then $\dot r$ and $\dot p \circ \dot q$ are homotopic.  
\end{proposition}
\begin{proof}
 It is enough to   verify the hypothesis of Lemma~\ref{lem:HomotopyQIpairs} for $\dot p\circ \dot q$ and $\dot r$.  Let $A\in G/\calp$. We show below that $\Hdist_K(\dot r(A), \dot p\circ \dot q(A))<\infty$. 

Suppose that $p$ is an $(L,C,M)$-quasi-isometry of pairs. The definition of $\dot p$ and $\dot q$ implies
\begin{equation}\label{eq:dotcomposition02}
\begin{split}
  \Hdist_K(\dot p \circ   \dot q(A),& p\circ q(A)) \leq  \\
  & \leq \Hdist_K(\dot p\circ   \dot q(A)  , p ( \dot q(A))  ) + \Hdist_K( p ( \dot q(A)) , p\circ   q(A))   \\
  & \leq 
  \Hdist_K(\dot p \circ   \dot q(A)  , p ( \dot q(A))  ) +  L \Hdist_H(   \dot q(A)  ,     q(A) ) +C \\
  & <\infty
\end{split}
\end{equation}
 Analogously,
 \begin{equation}\label{eq:dotcomposition01}
  \Hdist_K(\dot r(A), r(A))<\infty.   
 \end{equation}

 Combining \eqref{eq:dotcomposition02}, \eqref{eq:dotcomposition01}, and the fact that $r=p\circ q$ yields 
\begin{equation*}\label{eq:dotcomposition03}
 \Hdist_K(\dot r(A), \dot p\circ \dot q(A))
 \leq  \Hdist_K(\dot r(A), r(A)) +   
   \Hdist_K(p\circ q (A), \dot p\circ \dot q(A))  < \infty,  
\end{equation*}
which concludes the proof.
\end{proof}

We now turn to proof of the main result of this section, that the induced simplicial map is a homotopy equivalence.
 
\begin{proof}[Proof of Theorem~\ref{thm:HE}]
Let $q\colon (G,\calp)\to (H,\calq)$ be a quasi-isometry of pairs. A direct consequence of Lemma~\ref{lem:HomotopyQIpairs} is that all choices for the induced simplicial map $\dot q\colon \mathcal{K}(G,\calp) \to \mathcal{K}(H,\calq)$ are homotopic. 

Let $p\colon (H,\calq) \to (G,\calp)$ be a quasi-inverse of $q$. We need to prove that $\dot p\circ \dot q$ and $\dot q\circ \dot p$ are homotopic to the identity maps on $\mathcal{K}(G,\calp)$ and $\mathcal{K}(H,\calq)$ respectively. The arguments are analogous, so we only prove the statement for $\dot p\circ \dot q$.

Let $\mathbb{I}$ denote the identity map on $G$. Note that we can assume that $\dot{\mathbb{I}}$ is the identity map on $\mathcal{K}(G,\calp)$. 
Let $r=p\circ q$, and suppose that it is an $(L,C,M)$-quasi-isometry of pairs. By Proposition~\ref{prop:HomotopyComposition}, the induced maps $\dot r$ and $\dot p\circ \dot q$ are homotopic. Hence it is enough to show that $\dot r$ is homotopic to the identity map $\dot{\mathbb{I}}$ on $\mathcal{K}(G,\calp)$. 

Since $p$ is a quasi-inverse of $q$, there is a constant $D$ such that $r=p\circ q$ is at distance at most $D$ from the identity map $\mathbb{I}$ in the $L^\infty$--metric; that is,
\begin{equation}  \dist_G(r(x),\mathbb{I}_G(x))= \dist_G(r(x),x)\leq D
\end{equation}
for all $x\in G$.  Then we have that 
\begin{equation}
\begin{split}
\Hdist_G(\dot r(gP),\dot{\mathbb{I}}(gP))& = \Hdist_G(\dot r(gP), gP)\\ 
&   \leq 
 \Hdist_G ( \dot r(gP), r(gP))+ \Hdist_G(r(gP), gP)\\&\leq M+D <\infty .
\end{split}
\end{equation}
This verifies the hypothesis of Lemma~\ref{lem:HomotopyQIpairs} for $r$ and $\mathbb{I}$, and hence $\dot r$ is homotopic to the identity map on $\mathcal{K}(G,\calp)$. 
\end{proof}

\subsection{Coned-off Cayley graphs}\label{sec:coned-off}
 
Let $(G,\calp)$ be a group pair and    $S\subset G$ a finite relative generating set of $G$ with respect to $\calp$, that is,  $S\cup \bigcup \calp$ is a generating set of $G$. 

The \emph{ coned-off Cayley graph  of $G$ with respect to $\calp$} is the  $G$-graph $\widehat \Gamma(G,\calp,S)$ with vertex set $G\cup G/\calp$ and edge set consisting of
\begin{itemize}
    \item $\{g,gs\}$ for $g\in G$ and  $s\in S$,
    \item $\{g, gP\}$ for $g\in G$, $P\in \calp$.
\end{itemize}

Observe that the vertex set of the coset intersection complex $\mathcal{K}(G,\calp,S)$ is a subset of the vertex set of $\widehat \Gamma(G,\calp,S)$.  
The \emph{extended coned-off Cayley graph  of $G$ with respect to $\calp$} is the  $G$-graph $\widetriangle \Gamma(G,\calp,S)$ obtained by adding all edges of $\mathcal{K}(G,\calp,S)$ to $\widehat \Gamma(G,\calp,S)$, that is, 
\[  \widetriangle \Gamma(G,\calp,S) = \widehat \Gamma(G,\calp,S) \cup \mathcal{K}(G,\calp)^{(1)}. \]

The assumption that $S$ is a relative generating set implies that  $\widehat \Gamma(G,\calp,S)$ is   connected and that its  quasi-isometry type is independent of the finite set $S$. Moreover, since $S$ and $\calp$ are assumed to be finite sets, the $G$-action on $\widehat \Gamma(G,\calp,S)$ is cocompact.

In the case that the empty set is a relative generating set of $(G,\calp)$, we denote by $\widehat \Gamma(G,\calp)$ the coned-off Cayley graph for $S=\emptyset$;  analogously, $\widetriangle{\Gamma}(G,\calp)$ denotes the extended coned-off Cayley graph.

\begin{example}[Relatively hyperbolic groups]
In the case of a finitely generated group $G$ hyperbolic relative to a finite collection $\calp$, the extended coned-off Cayley graph of $(G,\calp)$ coincides with the coned-off Cayley graph since $\calp$ is an almost malnormal collection.   
\end{example} 

\begin{example}\label{ex:NonQI-extended}
The extended coned-off Cayley graph of $(\ZZ^2,\ZZ)$ with respect to any finite generating set is quasi-isometric to a point, but the action by $\ZZ^2$ on it is not cocompact since there are infinitely many distinct $G$-orbits of edges.      
\end{example} 

The quasi-isometry type  of the extended coned-off Cayley graph and the coset intersection complex  coincide in the following situation.

\begin{proposition}\label{prop:QI-coned-off}
    If $\mathcal{K}(G,\calp)$ is connected and $S$ is finite, then $ \mathcal{K}(G,\calp)$ and $\widetriangle\Gamma(G,\calp, S)$ are quasi-isometric.
\end{proposition}

The proposition is a consequence of  the following  well-known fact.

\begin{lemma}\label{lem:boring}
    Let $\Delta$ be a $G$-graph and $e$ one of its edges. If the $G$-graph $\Gamma$ obtained from $\Delta$ by removing all edges in the $G$-orbit of $e$ is connected, then the inclusion $\Gamma \hookrightarrow \Delta$ is a quasi-isometry of graphs.  
\end{lemma}

\begin{proof}
Let $u$ and $v$ be the endpoints of the edge $e$, and let $L$ be the distance in $\Gamma$ between $u$ and $v$. Let $x,y$ be two arbitrary vertices of $\Gamma$.
 Since $\Gamma$ is a subgraph of $\Delta$,  we have that $\dist_\Delta(x,y)\leq \dist_\Gamma(x,y)$. On the other hand, for any path in $\Delta$, one can replace each edge in the orbit of $e$ by a path in $\Gamma$ of length $L$, and hence $\dist_\Gamma(x,y)\leq L\dist_\Delta(x,y)$.  
\end{proof}

\begin{proof}[Proof of Proposition~\ref{prop:QI-coned-off}]
The graph $\widetriangle\Gamma(G,\calp)$ is obtained from the 1-skeleton of $\mathcal{K}(G,\calp)$ by attaching a  $G$-orbit of edges  for each $s\in S$ (with trivial stabilizers). Since $S$ is finite, 
the inclusion of the 1-skeleton of $ \mathcal{K}(G,\calp)$ into $\widetriangle\Gamma(G,\calp, S)$ is a quasi-isometry by Lemma~\ref{lem:boring}. 
\end{proof}

The quasi-isometry type of the coned-off Cayley graph and its extended version do not coincide in general; see Example~\ref{ex:NonQI-extended}. However, they do coincide for a class of group pairs arising from right-angled Artin groups, as we show in  Proposition~\ref{prop:RAAGsConedOff}.

 \section{Geometric properties of group pairs}\label{Sec:GeomProps}

In this section, we use the coset intersection complex to prove that several algebraic properties are quasi-isometry invariants of group pairs.  We discuss finite height and width  in Section~\ref{sec:HandW}.  We next consider two conditions which we call \textit{finite packing} and having a \textit{commensurated core} in Sections~\ref{sec:Packing} and \ref{sec:Core}, respectively.  Finally, we discuss thickness in Section~\ref{sec:thickness}.

\subsection{Height and Width}\label{sec:HandW}

We first  show that finite height and finite width, properties defined by Gitik, Mj, Rips, and Sageev in \cite{GMJRS98}, are quasi-isometry invariants of group pairs.

\begin{definition}[Height and width of subgroups]
Let $ \calp$ be a collection of subgroups of a group $G$.
\begin{enumerate}
    \item The \textit{height} of $\calp$ in $G$ is the maximal $m\in \NN$ such that  there exists a collection of $m$ distinct cosets $\{g_1P_1,\ldots ,g_mP_m\}\subset G/\calp$ such that the intersection $\bigcap_{i=1}^m g_iP_ig_i^{-1}$  is infinite.

    \item The \textit{width} of $\calp$ in $G$ is the maximal $m\in\NN$  such that there exists a collection of $m$ distinct cosets $\{g_1P_1,\ldots ,g_mP_m\}\subset G/\calp$ such that pairwise intersections  $g_iP_ig_i^{-1} \cap g_jP_jg_j^{-1}$  are infinite for all $1\leq i<j\leq m$.
\end{enumerate}
Note that finite width implies finite height. We say that the pair $(G,\calp)$ has finite height (resp. width) if $\calp$ has finite height (resp. width) in $G$.  
\end{definition}

Recall that a collection distinct cosets $\{g_1P_1,\ldots ,g_mP_m\}\subset G/\calp$ spans a simplex in  $\mathcal K(G,\calp)$ if and only if $\bigcap_{i=1}^m g_iP_ig_i^{-1}$  is infinite.  We immediately obtain the following lemma, which is Theorem~\ref{thm:alg&geom}\eqref{item:finitediml} and \eqref{item:cliques}.
 
\begin{lemma} Let $(G,\calp)$ be a group pair.
\begin{enumerate}
\item $\calp$ has \emph{height  $m\in\ZZ_+$} if and only if  $\mathcal{K}(G,\calp)$ has dimension $m$. 
\item $\calp$ has \emph{width  $m\in\ZZ_+$}  if and only if the largest clique in  the 1-skeleton of $\mathcal{K}(G,\calp)$ has $m$ vertices.  
\end{enumerate}
\end{lemma}


The main goal of this section is to prove Theorem~\ref{thm:GeomProps}\eqref{item:finiteheight} and \eqref{item:width}: 
\begin{theorem}\label{thm:main}
Having finite width (resp. finite height) is a quasi-isometry invariant of group pairs.    
\end{theorem}

The proof of the theorem relies on a sequence of lemmas.

\begin{lemma}\label{lem:HeightCommensurators2}
 A group pair with finite height  is reducible.   
\end{lemma}
\begin{proof}
If $(G,\calp)$ has finite height, then $\mathcal{K}(G,\calp)$ is finite dimensional. If $P$ has infinite index in $\comm_G(P)$, then   $\mathcal{K}(G,\calp)$ has simplices of arbitrarily large dimension.  Hence every $P\in\calp$ has finite index in its commensurator, and so  $(G,\calp)$ is reducible. 
\end{proof}

\begin{lemma}\label{lem:RefinementHeightWidth2}
 Suppose $(G,\calp)$ is reducible and let $(G,\calp^*)$ be a refinement. Then $(G,\calp)$ has finite height (resp. finite width) if and only if    $(G,\calp^*)$ has finite height (resp. finite width).
\end{lemma}
\begin{proof}
 Since $(G,\calp)$ is reducible, \cite[Proposition 6.3]{HMS2021} implies that the identity map on $G$ is a quasi-isometry of pairs $(G,\calp) \to (G,\calp^*)$. Since $(G,\calp^*)$ is reduced,   Proposition~\ref{lem:QI-to-refinedPair} implies that  there is a simplicial embedding $\dot p\colon \mathcal{K}(G,\calp^*) \to \mathcal{K}(G,\calp)$, and  
 a surjective simplicial map $\dot q\colon \mathcal{K}(G,\calp) \to \mathcal{K}(G,\calp^*)$ that is uniformly finite-to-one at the level of zero skeletons.   Suppose that $\dot q^{-1}(gP)$ has cardinality at most $n$ for every $gP\in G/\calp$.  

Using the embedding $\dot p$, we observe that the maximal  
number of vertices of a clique in $\mathcal{K}(G,\calp)$  is an upper bound   for the maximal number of vertices of a clique in $\mathcal{K}(G,\calp^*)$. Hence if $(G,\calp)$ has finite width then $(G,\calp^*)$ has finite width. The argument for finite height is analogous, replacing ``number of vertices in a clique" with ``dimension of a simplex."

For the other implication, observe that $\dot q$ maps cliques to cliques. Suppose that   $d$ is an upper bound for the maximal number of vertices of a clique in $\mathcal{K}(G,\calp^*)$. Note that if a clique of $\mathcal{K}(G,\calp)$ has $m$ vertices, then its image in $\mathcal{K}(G,\calp^*)$ is a clique with at least $m/n$ vertices. Since $m/n\leq d$, it follows that $m\leq d\cdot n$. Hence if $(G,\calp^*)$ has finite width, then $(G,\calp)$ has finite width. The argument for height is again completely analogous.    
\end{proof}

\begin{lemma}\label{lem:HeightWidthConclusion2}
Let $(G,\calp)$ and $(H,\calq)$ be quasi-isometric pairs. 
Suppose that $(G,\calp)$ is reducible. Then   $(G,\calp)$ has finite height (resp. finite width) if and only if    $(H,\calq)$ has finite height (resp. finite width).  
\end{lemma}
\begin{proof}
By Theorem~\ref{thm:reducible},  being reducible is a quasi-isometric invariant, and hence     $(H,\calq)$ is reducible. 
Let $(G,\calp^*)$ and $(H,\calq^*)$ be   refinements of $(G,\calp )$ and $(H,\calq)$ respectively. Then \cite[Proposition 6.3]{HMS2021} implies that the pairs $(G,\calp)$, $(H,\calq)$,  $(G,\calp^*)$ and $(H,\calq^*)$ are all quasi-isometric pairs.  By Remark~\ref{rem:RefinementsRedued2}, $(G,\calp^*)$ and $(H,\calq^*)$ are quasi-isometric reduced pairs. Then Proposition~\ref{prop:IsomorphicCosetIntersectionCmplex2} implies that
the coset intersection complexes $\mathcal{K}(G,\calp^*)$ and $\mathcal{K}(H,\calq^*)$ are isomorphic. Therefore 
$(G,\calp^*)$ has finite height (resp. width) if and only if $(H,\calq^*)$ has finite height (resp. width). Then the conclusion follows from Lemma~\ref{lem:RefinementHeightWidth2}.   
\end{proof}

We are now ready to prove Theorem~\ref{thm:main}.
\begin{proof}[Proof of Theorem~\ref{thm:main}] Suppose that $(G,\calp)$ and $(H,\calq)$ are a quasi-isometric  pairs and  that $(G,\calp)$ has finite height or finite width. Since finite width implies finite height, in either case,   $(G,\calp)$ has finite height.  Lemma~\ref{lem:HeightCommensurators2} then implies that  $(G,\calp)$ is reducible, and the conclusion of the theorem follows directly from Lemma~\ref{lem:HeightWidthConclusion2}. 
\end{proof}

Huang and Wise introduced the notion of \textit{finite stature} for a group pair $(G,\calp)$, a property closely related to finite height~\cite{huang2019}.  To encode stature using the coset intersection complex $\mc K(G,\calp)$, one needs to record information not just about the simplices, as was the case for  height, but also about the action of elements of $\calp$ on $\mc K(G,\calp)$.  Additional assumptions would be needed to apply the methods of this paper to show that finite stature is geometric.

\begin{question}
    Is having finite stature a geometric property of group pairs?
\end{question}

\subsection{Finite packing (cocompactness)}\label{sec:Packing}
In this subsection, we introduce the notion of finite packing.

\begin{definition}[Finite packing]\label{def:NewFinitePacking} A    collection $\calp$ of subgroups of a group $G$ has \emph{finite packing} if there is a finite subset of left cosets  $D$ in $G/\calp$ such that for any finite subset ${g_1P_1,\ldots , g_dP_d}$ of $G/\calp$, if $\bigcap_{i=1}^d g_iP_ig_i^{-1}$ is  infinite,  then there is $g\in G$ such that $gg_iP_i\in D$ for each $i\in\{1,\ldots,d\}$. \end{definition}

\begin{remark}[Cocompactness] \label{rem:cocompactness} A group pair $(G,\calp)$ has {finite packing} if and only if the $G$-action on the coset intersection complex $\mathcal{K}(G,\calp)$ has finitely many $G$-orbits of simplicies, or  equivalently, the $G$-action on $\mathcal{K}(G,\calp)$ is cocompact. \end{remark}


Note that   $(G,\calp)$ has finite packing if and only if   $\mathcal{K}(G,\calp)$ is finite dimensional and has finitely many $G$-orbits of $d$-simplices for every $d$. In particular,  finite packing implies finite height.

\begin{proposition}\label{prop:finitePackingChar01}
If $(G,\calp)$ is a group pair, the following statements are equivalent.
\begin{enumerate}
    \item $(G,\calp)$ has finite packing. 
    \item There is $r>0$ such that for any simplex $\{g_1P_1,\ldots ,g_kP_k\}$  of $\mathcal{K}(G,\calp)$, the intersection $\bigcap_{i=1}^k \mathcal{N}_r(g_iP_i)$ is non-empty.
\end{enumerate}
\end{proposition}
\begin{proof}
Suppose (2) holds. Let $\cals$ be an arbitrary simplex of $\mathcal{K}(G,\calp)$, and let $\cals'=\{g_1P_1,\ldots ,g_kP_k\}$ be a finite subset of $\cals$.  Let $g\in \bigcap_{i=1}^k  \mathcal{N}_r(g_iP_i)$. Then there are elements $\{h_i\mid i=1,\dots, k\}$  of $G$ such that $gh_i P_i=g_i P_i$ and $\dist_G(1,h_i)\leq r$, so that $g^{-1}\cals' = \{h_iP_i\mid i=1,\dots, k\}$.  
Since  $\dist_G$ is locally finite (as $G$ is finitely generated), then there is a uniform bound $M$ on the number of distinct possibilities for $h_i$. As  $\calp$ is finite,  $g^{-1}\cals'$ contains at most $M\cdot |\calp|$ distinct elements. In particular, $\cals$ is a finite simplex, and there are finitely many possibilities for $g^{-1}\cals$.  Thus $(G,\calp)$ has finite packing. 

Conversely, suppose that (1) holds. Then $\mathcal{K}(G,\calp)$ is finite dimensional.   For each simplex $\cals=\{g_1P_1,\ldots, g_kP_k\}$, let \[ \mu(\cals)=\inf\left\{m \mid \bigcap_{i=1}^k \mathcal{N}_m(g_iP_i) \neq \emptyset \right\}. \]
Observe that $\mu(\cals)$ is well-defined since  
 $ \bigcap_{i=1}^k g_iP_ig_i^{-1} \subseteq \bigcap_{i=1}^k \mathcal{N}_m(g_iP_i)$
for any $m>\max_{i=1,\dots, k} \{\dist_G(1,g_i)\}$, and hence $\mu(\cals)$ is an infimum of non-empty set. On the other hand,     $\mu$ is $G$-invariant, that is, $\mu (g\cals) =\mu (\cals)$. Since there are finitely many $G$-orbits of simplices,
\[  r =\sup\left\{ \mu(\cals) \mid \text{$\cals$ is a simplex of $\mathcal{K}(G,\calp)$}\right\} < \infty\]
 satisfies the required property.
\end{proof}

\begin{example}[Quasiconvex subgroups of hyperbolic groups have finite packing] \label{ex:HypQc}
 If $G$ is a hyperbolic group and $\calp$ is a finite collection of infinite quasi-convex subgroups, then $(G,\calp)$ has finite packing.  The argument by  Gitik, MJ, Rips and Sageev  in~\cite{GMJRS98} proving that $(G,\calp)$ has finite width actually shows that Proposition~\ref{prop:finitePackingChar01} (2) holds. Let us briefly recall the argument. Assume that the Cayley graph $\Gamma$ of $G$ has $\delta$-slim triangles and that all subgroups in $\calp$ are $\sigma$-quasi-convex.  Suppose $\{g_1P_1,\ldots ,g_kP_k\}$ is a simplex $\mathcal{K}(G,\calp)$.  Then there is $g\in \bigcap_{i=1}^k g_iP_ig_i^{-1}$ such that $\dist_G(e,g)>3\max_i\{\dist_G(e,g_i)\}+3\delta$. Let $\gamma=[e,g]$ be a geodesic in $\Gamma$ from $e$ to $g$, and let $x$ be the middle vertex of $\gamma$. Consider the geodesic rectangle in $\Gamma$ with geodesic sides $[e,g_i]$,   $[g_i,gg_i^{-1}]$, $[gg_i^{-1}, g]$ and $\gamma$. Since rectangles are $2\delta$-slim, and $\gamma$ is long enough, the vertex $x$ is at distance $2\delta$ from a point in  $[g_i,gg_i^{-1}]$. This geodesic has endpoints in the $\sigma$-quasiconvex left coset $g_iP_i$, and hence $x$ is at distance $r=2\delta+\sigma$ from a vertex in $g_iP_i$. It follows that $x\in \bigcap_{i=1}^k\mathcal{N}_r (g_iP_i)$.  By Proposition~\ref{prop:finitePackingChar01}, $(G,\calp)$ has finite packing.
\end{example}

\begin{proposition}\label{prop:FinitePackingQI}
 Having finite packing is a quasi-isometry invariant of group pairs. 
\end{proposition}
\begin{proof}
Let $q\colon (G,\calp) \to (H,\calq)$ be an $(L,C,M)$-quasi-isometry of group pairs and suppose that $(H,\calq)$ has finite packing.  
We use the characterization of finite packing provided by Proposition~\ref{prop:finitePackingChar01} to prove that $(G,\calp)$ has finite packing.

Let $\dot q\colon \mathcal{K}(G,\calp) \to \mathcal{K}(H,\calq)$ be an  induced  simplicial map;  see Proposition~\ref{prop:SimplicalCosetIntersection02}. Recall that $\Hdist_G(\dot q (gP), q(gP))\leq M$ for all $gP\in G/\calp$.  Since $(H,\calq)$ has finite packing, the complex $\mc K(H,\calq)$ is finite dimensional and there is $r>0$ such that if $\{h_1Q_1,\ldots ,h_kQ_k\}$  is a simplex of $\mathcal{K}(H,\calq)$ then  $\bigcap_{i=1}^k \mathcal{N}_r(h_iQ_i)$ is non-empty. Let $r'=L(r+M+2C)$, and let $\{g_1P_1,\ldots, g_kP_k\}$ be an arbitrary  simplex of $\mathcal{K}(G,\calp)$. Let $h_iQ_i=\dot q(g_iP_i)$, and let $h\in \bigcap_{i=1}^k \mathcal{N}_r(h_iQ_i)$. Then there exists $g\in G$ such that $\dist_H(q(g),h)\leq C$. Since $\dist_H(h, h_iQ_i)\leq r$, it follows that 
\[ 
\begin{split}
\dist_G(g, g_iP_i)&\leq L\dist_H(q(g),q(g_iP_i))+ LC  \\   
&\leq L\dist(h, q(g_iP_i)) + 2LC \\
& \leq
L(\dist_H(h, h_iQ_i) +\dist_H(h_iQ_i,q(g_iP_i) )  +2LC \\ 
&\leq Lr+LM+2LC=r'.
\end{split}
 \] Thus $\bigcap_{i=1}^k \mathcal{N}_{r'} (g_iP_i) \neq \emptyset$, and so the pair $(G,\calp)$ has finite packing by Proposition~\ref{prop:finitePackingChar01}.    
\end{proof}

\begin{question}
    Let $(G,\calp)$ be a group pair.  Suppose that $\mathcal{K}(G,\calp)$ has finitely many $G$-orbits of $d$-simplices for all $d$. Is $\mathcal{K}(G,\calp)$   $G$-cocompact?
\end{question}


\subsection{Commensurated cores}\label{sec:Core}
In this section, we introduce the notion of a commensurated core of a group pair.  After giving geometric and algebraic defintions of a commensurated core in Definition~\ref{def:CommCore} and Proposition~\ref{prop:MalnormalCoreChar}, respectively, we turn to the main  result of the section, Proposition~\ref{prop:FiniteCoreQI}, which uses commensurated cores and height to describe exactly when the action of $G$ on $\mc K(G,\calp)$ is cocompact.

\begin{definition}[Commensurated  core]\label{def:CommCore}
Let $(G,\calp)$ be a group pair. A \emph{commensurated core} $\calq$ of $(G,\calp)$ is a collection of   subgroups representing the conjugacy classes of setwise $G$-stabilizers of maximal simplices in the coset intersection complex $\mathcal{K}(G,\calp)$.  
\end{definition}

Observe that a maximal simplex  of $\mathcal{K}(G,\calp)$ is a maximal subset $\mathcal S$ of $G/\calp$ such that  for any finite subset $\{g_1P_1, \ldots , g_kP_k\} \subset \mathcal{S}$ the subgroup $\bigcap_{i=1}^k g_iP_i g_i^{-1}$ is infinite. Zorn's lemma guarantees that every simplex of $\mathcal{K}(G,\calp)$ is a face of a maximal simplex,  but this maximal simplex might be infinite dimensional. 
Moreover, the collection of maximal simplices of $\mathcal{K}(G,\calp)$ is naturally a $G$-set, and it turns out that the isotropies of this $G$-set are commensurated subgroups of $G$ as we explain below. 
     

Given a collection $\cals\subset G/\calp$, let \
\begin{equation}\label{eqn:K(S)}
K(\mathcal{S})=\bigcap \left\{gPg^{-1} \mid gP\in \cals \right\}.
\end{equation}
If $\cals $ is a simplex of $\mc K(G,\calp)$, then $K(\cals)$ is the pointwise stabilizer of the simplex $\cals$.  If $K(\cals)$ is an infinite subgroup of $G$, then $\cals$ is a simplex of $\mc K(G,\calp)$.   However, the converse does not always hold. 
 If $\cals$ is a finite-dimensional simplex, then $K(\cals)$ is always infinite.  On the other hand, if $\cals$ is an infinite-dimensional simplex, then  $K(\cals)$ may be a finite group; see Example~\ref{ex:BS12-2}.

We are interested in group pairs $(G,\calp)$ with the property that pointwise stabilizers of maximal simplices of $\mathcal{K}(G,\calp)$ are infinite subgroups, that is, $K(\cals)$ is infinite for all maximal simplices $\cals$. An algebraic characterization of this property is as follows: for any collection $\cals\subset G/\calp$ such that $K(\cals')$ is infinite for any finite $\cals'\subseteq \cals$, we have $K(\cals)$ is infinite.  If this property holds for a group pair $(G,\calp)$, we say the pair satisfies the \textit{infinite intersection property}.

\begin{example}\label{ex:BS12-2}
Consider the Baumslag-Solitar group $B(1,k)=\langle a,t\mid tat^{-1}=a^k$. The group pair $(BS(1,k),\langle a \rangle)$ does not satisfy the infinite intersection property. For instance, $\cals=\{ t^n\langle a\rangle \mid n\in\ZZ\}$ is an infinite dimensional simplex of $\mathcal{K}(BS(1,k),\langle a \rangle)$ with trivial pointwise stabilizer. 
\end{example}

In Lemma~\ref{lem:heightSimplices} below, we point out that group pairs with finite height satisfy the infinite intersection property.   On the other hand, there are interesting group pairs with infinite height that also satisfy the infinite intersection property, for example  pairs consisting of a right-angled Artin group together with maximal standard abelian subgroups, as described in
Section~\ref{sec:RAAGs}.

\begin{proposition}[Commensurated core, algebraic characterization]\label{prop:MalnormalCoreChar}
 Let $(G,\calp)$ be a group pair that satisfies the infinite intersection property. 
 Then a collection $\calq$ of subgroups of $G$   is a commensurated core if and only if the following three conditions hold.
 \begin{enumerate}
\item If $\cals \subset G/\calp$ is  maximal with respect to the property that  $K(\cals)$  is  an infinite subgroup, then $\comm_G(K(\cals))$ is conjugate in $G$ to some $Q\in\calq$.

 \item If $Q\in\calq$, then there is a collection $\cals\subset G/\calp$ that is maximal with respect to the property that $K(\cals)$ is infinite and satisfies $ Q= \comm_G(K(\cals))$.
\item No two distinct subgroups in $\calq$  are conjugate in $G$.
 \end{enumerate}
\end{proposition}

Before the proof of the proposition, we need a lemma.

\begin{lemma}\label{lem:SetWiseStabilizers}
Let $(G,\calp)$ be a group pair, and let $\cals \subset G/\calp$  a maximal simplex of $\mathcal{K}(G,\calp)$ such that $K(\mathcal{S})$ is an infinite subgroup.
Then $\comm_G(K(\cals))$ is the setwise $G$-stabilizer of the simplex $\cals$.
\end{lemma}

\begin{proof}
Suppose  $g\in G$  fixes the simplex $\cals=\{g_iP_i\mid i\in I\}$ setwise. Then the simplex  $g\cals=\{gg_iP_i \mid g_iP_i\in\mathcal{S}\}$ equals the simplex $\cals$, and so  $K(\cals)=K(g\cals)=gK(\cals)g^{-1}$.  Hence $g$ is in the normalizer of $K(\cals)$ and, in particular, in the commensurator $\comm_G(K(\cals))$.

Conversely, suppose that $g\in \comm_G(K(\cals))$. Then  the subgroup $K(g\cals)$   is commensurable to $K(\cals)$, and hence $K(\cals)\cap K(g\cals)$ is  infinite. It follows that $\cals\cup g \cals$ is a simplex of $K(G,\calp)$. Since $\cals$ is a maximal simplex, we have that $g\cals$ is a maximal simplex as well. By maximality,  $g\cals = \cals$ and so $g$ stabilizes $\cals$ setwise.         
\end{proof}

\begin{proof}[Proof of Proposition~\ref{prop:MalnormalCoreChar}]
Since $(G,\calp)$ satisfies the infinite intersection property, a subset $\cals \subset G/\calp$ is a maximal simplex if and only if $\cals$ is maximal with respect to the  property that $K(\cals)$ is infinite. Thus when $\cals$ is a maximal simplex,  Lemma~\ref{lem:SetWiseStabilizers} implies that $\comm_G(K(\cals))$ is the setwise stabilizer of $\cals$.

Using these facts, we see that Property (1) holds if and only if $\calq$ contains representatives of all conjugacy classes of setwise stabilizers of maximal simplices of $\mc K(G,\calp)$.  Similarly, Property (2) holds if and only if every subgroup in $\calq$ is the setwise stabilizer of a maximal simplex of $\mc K(G,\calp)$.  Therefore, $\calq$ satisfies (1), (2), and (3) if and only if it is a commensurated core.
\end{proof}

It was pointed out by Agol, Groves and Manning that for pairs $(G,P)$ where $G$ is a hyperbolic group and $P$ is a quasi-convex subgroup, the commensurated core is an almost malnormal collection~\cite{AGM09}. The following proposition generalizes their result. 

\begin{proposition} \label{prop:HeightMalnormalCore} 
 A commensurated core of a group pair with finite height is an almost malnormal collection.    
\end{proposition}

The proof of this proposition relies on the following lemma which, in particular, shows that if a pair $(G,\calp)$ has finite height ,then pointwise stabilizers of maximal simplices of $\mathcal{K}(G,\calp)$ are infinite.  Therefore the commensurated core of pairs with finite height can be characterized by Proposition~\ref{prop:MalnormalCoreChar}.

\begin{lemma}\label{lem:heightSimplices}
Suppose $(G,\calp)$ has finite height, and let $\cals \subset G/\calp$ be a maximal simplex of $\mathcal{K}(G,\calp)$. Then 
\begin{enumerate}
    \item $K(\cals)$ is infinite  and is the pointwise $G$-stabilizer of the simplex $\cals$.
    \item $\cals$ is a maximal subset of $G/\calp$ with the property that $K(\cals)$ is infinite.
    \item $K(\cals)$ is finite index in $\comm_G(K(\cals))$.
\end{enumerate}
\end{lemma}
\begin{proof}  Since $(G,\calp)$ has finite height, the simplex $\cals$ is a finite subset of $G/\calp$ and, in particular, $K(\cals)$ is an infinite subgroup. The first statement is then immediate,  
while the second statement re-states that $\cals$ is a maximal simplex. Since $\cals$ is a finite subset of $G/\calp$, its pointwise stabilizer $K(\cals)$ is a finite index subgroup of its setwise stabilizer which is $\comm_G(K(\cals))$ by Lemma~\ref{lem:SetWiseStabilizers}.     \end{proof}

\begin{proof}[Proof of Proposition~\ref{prop:HeightMalnormalCore}]
Let $Q,Q'\in \calq$ and $g\in G$.  There are maximal simplices $\cals$ and $\cals'$   of $\mathcal{K}(G,\calp)$ with setwise stabilizers $gQg^{-1}$ and $Q'$  respectively.    Suppose    $gQg^{-1}\cap  Q' =\comm_G(K(\cals)) \cap \comm_G(K(\cals'))$ is an infinite subgroup. By the third item of  Lemma~\ref{lem:heightSimplices},  we have that $K(\cals)\cap K(\cals') = K(\cals\cup\cals')$ is infinite.  By maximality, it follows that $\cals=\cals'$ and hence $gQg^{-1}=Q'$. It follows that $Q=Q'$ by definition of $\calq$, and so $g\in \comm_G(Q)=Q$.
\end{proof}

\begin{example}[Commensurated cores of  quasi-convex subgroups in hyperbolic groups are almost malnormal]\label{ex:hypQcMalnormalCore} Let $G$ be a hyperbolic group and $\calp$  a collection of infinite quasiconvex subgroups. Then  $(G,\calp)$ has finite height; see Example~\ref{ex:HypQc}. Therefore, by Proposition~\ref{prop:HeightMalnormalCore}, if $\calq$ is a commensurated core of $(G,\calp)$, then $\calq$ is an almost malnormal collection of subgroups of $G$, recovering a result from Agol, Groves and Manning~\cite{AGM09}.    In particular $\mathcal{K}(G,\calq)$ is a $0$-dimensional simplical $G$-complex.    In this case, the commensurated core coincides with the notion of the malnormal core  from~\cite{AGM09}. 
  
\end{example}

The following proposition exhibits a  relationship between the new properties that we have introduced: finite packing and a finite commensurated core. It also proves Theorem~\ref{thm:alg&geom}\eqref{item:Gcocompact}, as finite packing is equivalent to   cocompactness of the coset intersection complex; see Remark~\ref{rem:cocompactness}. 


\begin{proposition}\label{prop:FiniteCoreQI}
A group pair $(G,\calp)$ has finite packing if and only if   $(G,\calp)$ has finite height and a finite commensurated   core. 
\end{proposition}

We begin with a lemma.
\begin{lemma}\label{lem:coreQI}
If $\mathcal{K}(G,\calp)$ is finite dimensional, then the number of $G$-orbits of maximal simplices is finite if and only if the number of conjugacy classes of $G$-setwise stabilizers of maximal simplices is finite.    
\end{lemma}
\begin{proof}
 The only if part of the lemma is trivial. For the if part, first 
recall that $\mathcal{K}(G,\calp)$ is finite dimensional if and only if $(G,\calp)$ has finite height. 
By Lemma~\ref{lem:heightSimplices},   if $\Delta$ and $\Delta'$ are maximal simplices with the same setwise stabilizer, then their pointwise stabilizers are commensurable and then maximality implies that $\Delta=\Delta'$. Hence if two maximal simplices have conjugate setwise $G$-stabilizers, then they are in the same $G$-orbit, and the conclusion follows. 
\end{proof}

\begin{proof}[Proof of Proposition~\ref{prop:FiniteCoreQI}]
The pair $(G,\calp)$ having finite height is equivalent to $\mathcal{K}(G,\calp)$ being finite dimensional. 
Hence it is enough to prove that if $\mathcal{K}(G,\calp)$ is finite dimensional, then $\mathcal{K}(G,\calp)$ is $G$-cocompact if and only if there are finitely many conjugacy classes of setwiset stabilizers of maximal simplices. Since a finite dimensional simplicial $G$-complex is cocompact if and only if there are finitely many $G$-orbits of maximal simplices, the statement of the proposition follows from Lemma~\ref{lem:coreQI}. 
\end{proof}

If $\mc K(G,\calp)$ is infinite-dimensional, then the $G$-action is not cocompact.  However, the converse is unknown.  

\begin{question}
Suppose $\mathcal{K}(G,\calp)$ is finite dimensional. Is the $G$-action on $\mathcal{K}(G,\calp)$ cocompact?  Equivalently, is there an example of a group pair $(G,\calp)$ that has finite height but does not have a finite commensurated core?
\end{question}

\subsection{Fenced infinite intersections}

\begin{definition}\label{def:FencedInfInt}
Let $\tau\geq 0$. A group pair $(G,\calp)$ has   \emph{$\tau$-fenced infinite intersections} 
 if for any $g_1P_1, g_2P_2 \in G/\calp$  such that   $g_1P_1g_1^{-1} \cap g_2P_2g_2^{-1}$ is an infinite subgroup,   $\dist_G(g_1P_1, g_2P_2)\leq \tau$.  If $(G,\calp)$ has $\tau$-fenced infinite intersections for some $\tau$, we  say that
 $(G,\calp)$ has \emph{ fenced infinite intersections}.
\end{definition}

\begin{lemma}\label{lem:tauEdges}
Let $(G,\calp)$ be a group pair. For any $\tau\geq 0$,  
\[  E_\tau(G,\calp) = \left\{  \{g_1P_1,g_2P_2\}\mid g_1P_1,g_2P_2\in G/\calp \text{ and } 0<\dist_G(g_1P_1,g_2P_2
)\leq \tau  \right\}  \]
is a $G$-set with $G$-action  given by $g. \{g_1P_1,g_2P_2\} = \{gg_1P_1, gg_2P_2\}$.  Moreover, $E_\tau(G,\calp)$ has only finitely many distinct $G$-orbits. 
\end{lemma}
\begin{proof}
That $E_\tau(G,\calp)$ is a $G$-set follows from the  observation that   $\dist_G( g_1P_1, g_2P_2) =  \dist_G( gg_1P_1, gg_2P_2)$ for any  $\{g_1P_1, g_2P_2\}\in E_\tau(G,\calp)$ and any $g\in G$. On the other hand,  any $G$-orbit  has a representative of the form $\{P_1, g_2P_2\}$ such that $\dist_G(P_1, g_2P_2)=\dist_G(e, g_2)$. Since 
$(G,\dist_G)$ is locally finite and $\calp$ is finite, there are finitely many choices for triples  $(P_1, g_2, P_2) \in \calp\times G\times \calp$ such that $\dist_G(P_1, g_2P_2)=\dist_G(e, g_2)\leq \tau$. Therefore $E_\tau(G,\calp)$ has only finitely many distinct $G$-orbits.  
\end{proof}

The following proposition proves Theorem~\ref{thm:alg&geom}\eqref{item:Gcocompact1skel}.
\begin{proposition}[Geometric characterization of fenced infinite intersections]
A group pair $(G,\calp)$ has fenced infinite intersections if and only if the $G$-action on the $1$-skeleton of $\mathcal{K}(G,\calp)$ is cocompact.   
\end{proposition}
\begin{proof}
By definition $\{g_1P_1, g_2P_2\}$ is a $1$-cell of   $\mathcal{K}(G,\calp)$ if and only if the intersection $g_1P_1g_1^{-1}\cap g_2P_2g_2^{-1}$ is infinite. If $(G,\calp)$ has $\tau$-fenced intersections, then the $G$-set of $1$-cells of $\mathcal{K}(G,\calp)$ is a $G$-subset of the $G$-set $E_\tau(G,\calp)$ which has only  finitely many distinct $G$-orbits by  Lemma~\ref{lem:tauEdges}. Since $\calp$ is finite, $\mathcal{K}(G,\calp)$ also has only finitely many distinct $G$-orbits of $0$-cells, and so its $1$-skeleton is $G$-cocompact.   Conversely, if there are finitely many $G$-orbits of $1$-cells, then we let $\tau$ be the  maximum of $\dist_G( g_1P_1, g_2P_2)$ over all possible $1$-cells $\{g_1P_1, g_2P_2\}$ and observe that $(G,\calp)$ has $\tau$-fenced infinite intersections.   
\end{proof}

The following theorem proves Theorem~\ref{thm:GeomProps}\eqref{item:fencedint}.
\begin{theorem}\label{thme:FIIproperty}
Having fenced infinite intersections is a quasi-isometry invariant of group pairs.     
\end{theorem}

\begin{proof}
Suppose that $q\colon (G,\calp) \to (H,\calq)$ is a $(L,C,M)$-quasi-isometry of pairs and $(H,\calq)$ has $\tau$-fenced infinite intersections. Consider an induced function $\dot q\colon G/\calp \to H/\calq$ given by  Proposition~\ref{propF:InducedDot-q-03}.

Let $g_1P_1, g_2P_2 \in G/\calp$ be such that $g_1P_1g_1^{-1} \cap g_2P_2g_2^{-1}$ is an infinite subgroup, and let $h_iQ_i=\dot q(g_iP_i)$ for $i=1,2$. 
By Lemma~\ref{lem:simplices}, $h_1Q_1h_1^{-1}\cap h_2Q_2h_2^{-1}$ is infinite. Since $(H,\calq)$ has $\tau$-fenced infinite intersections, we have that $\dist_H(h_1Q_1,h_2Q_2)\leq \tau$. By  Proposition~\ref{propF:InducedDot-q-03}, we have that 
$\dist_G(g_1P_1, g_2P_2)\leq L\tau + (C+2M)L$. Hence $(G,\calp)$ has $(L\tau + (C+2M)L)$-fenced infinite intersections. 
\end{proof}

\subsection{Networks and thickness}\label{sec:thickness}

\begin{definition}[$\tau$-Network]
Let $X$ be a metric space and $\mathcal{L}$ a collection of subsets of $X$. Given $\tau\geq0$, the space $X$ is a $\tau$-network with respect to $\mathcal{L}$ if:
\begin{enumerate}
\item  $X=\bigcup_{L\in \mathcal{L}} \mathcal{N}_\tau(L)$, and
\item  For any $L,L'\in \mathcal{L}$ there exists a sequence  $L_1 = L, L_2, \ldots ,  L_n=L  $ of elements of $\mathcal{L}$ such that 
$\mathcal{N}_\tau (L_i)\cap \mathcal{N}_\tau (L_{i+1})$ is an unbounded subspace for all $i$.
\end{enumerate}
\end{definition}

\begin{definition}
A group pair $(G,\calp)$ is a $\tau$-network if $\calp$ is non-empty and $(G,\dist_G)$ is a $\tau$-network with respect to $G/\calp$.  
\end{definition}

\begin{definition}[The $\tau$-coset intersection complex $\mathcal{K}_\tau(G,\calp)$]
Let $(G,\calp)$ be a group pair and let $\tau\geq0$. The \emph{$\tau$-coset intersection complex} $\mathcal{K}_\tau(G,\calp)$ is the simplicial $G$-complex with vertex set $G/\calp$ and  such that a finite subset $\{g_1P_1,\ldots , g_kP_k\}$ of $G/\calp$ defines a simplex if and only if $\bigcap_{i=1}^k\mathcal{N}_\tau(g_iP_i)$ is unbounded in $(G,\dist_G)$.
\end{definition}

The $\tau$-coset intersection complexes define a filtration indexed by $\RR_+$ of the coset intersection complex. The next lemma shows that each term in this filtration is a $G$-complex with cocompact $1$-skeleton.

\begin{lemma}
Let $(G,\calp)$ be a group pair and $\tau\geq0$.
\begin{enumerate}
\item The  complex $\mathcal{K}_\tau(G,\calp)$ is a $G$-subcomplex of $\mathcal{K}(G,\calp)$.
\item The $1$-skeleton of $\mathcal{K}_\tau(G,\calp)$ is $G$-cocompact.
\end{enumerate}
\end{lemma}
\begin{proof}
That $\mathcal{K}_\tau(G,\calp)$ is a $G$-subcomplex
of $\mathcal{K}(G,\calp)$ is  a consequence of the metric $\dist_G$ being $G$-equivariant and  Lemma~\ref{prop:GeometryCosetIntersectionComplex2}, which characterizes simplices of the coset intersection complex $\mathcal{K}(G,\calp)$ geometrically.  For the second statement, observe that the $G$-set of edges of $\mathcal{K}_\tau(G,\calp)$ is a $G$-subset of the $G$-set $E_\tau(G,\calp)$ of Lemma~\ref{lem:tauEdges}, which has only finitely many distinct $G$-orbits. Hence  of $\mathcal{K}_\tau(G,\calp)$ has only finitely many $G$-orbits of edges and, since $\calp$ is finite, only finitely many $G$-orbits of vertices. Therefore the $1$-skeleton of $\mathcal{K}_\tau(G,\calp)$ is $G$-cocompact.
\end{proof}

\begin{proposition} \label{prop:NetworkChar}
Let $(G,\calp)$ be a group pair with $\calp$ a non-empty collection.
Then  $(G,\calp)$ is a $\tau$-network  if and only if $\mathcal{K}_\tau(G,\calp)$ is  connected.
\end{proposition}
\begin{proof}
A sequence $g_1P_1,  \ldots ,  g_nP_n$ in $G/\calp$ is the sequence of vertices of an edge-path in $\mathcal{K}_\tau(G,\calp)$ if and only if   $\mathcal{N}_\tau (g_iP_i)\cap \mathcal{N}_\tau (g_{i+1}P_{i+1})$ is an unbounded subspace for  $1\leq i <n$.  Therefore $\mathcal{K}_\tau(G,\calp)$ is connected if and only if for any two $gP,g'P'$ there is a sequence $gP=g_1P_1, g_2P_2  \ldots ,  g_nP_n=g'P'$ in $G/\calp$ such that $\mathcal{N}_\tau (g_iP_i)\cap \mathcal{N}_\tau (g_{i+1}P_{i+1})$ is unbounded for $1\leq i<n$. 

Since $\calp$ is non-empty, it is always the case that $G=\bigcup_{L\in G/\calp} \mathcal{N}_{\tau}(L)$. Therefore $(G,\dist_G)$ is a $\tau$-network with respect to $G/\calp$ if and only if 
$\mathcal{K}_\tau(G,\calp)$ is connected.  
\end{proof}

A  consequence of Proposition~\ref{prop:NetworkChar} is the following  characterization of being a network, which proves Theorem~\ref{thm:alg&geom}\eqref{item:connected}.
\begin{corollary} \label{lem:NetworkChar2}
  Let $(G,\calp)$ be a group pair.  Then $(G,\dist_G)$ is a network with respect to the collection $G/\calp$ if and only if $\calp$ is non-empty and $\mathcal K(G,\calp)$ has a connected and $G$-cocompact subgraph $\Delta$ with vertex set $G/\calp$.   
\end{corollary}
\begin{proof}
If $(G,\calp)$ is a $\tau$-network, then $\calp$ is non-empty, and so Proposition~\ref{prop:NetworkChar} implies that the $1$-skeleton of $\mathcal{K}_\tau(G,\calp)$ is a connected and cocompact $G$-subgraph of $\mathcal K(G,\calp)$ with vertex set $G/\calp$.   

Conversely, suppose that $\calp$ is non-empty and $\Delta$ is a connected and $G$-cocompact subgraph of $\mathcal K(G,\calp)$ with vertex set $G/\calp$.  Let $e_1,\ldots, e_k$ be a collection of representatives of $G$-orbits of edges of $\Delta$. By Lemma~\ref{prop:GeometryCosetIntersectionComplex2}, for each $e_i=\{g_iP_i, g_i'P_i'\}$ there is $\tau_i\geq 0$  such that $\mathcal{N}_{\tau_i}(g_iP_i)\cap \mathcal{N}_{\tau_i}(g_i'P_i')$ is unbounded in $(G,\dist_G)$. Let $\tau=\max\{\tau_1,\ldots ,\tau_k\}$. By $G$-equivariance of $\dist_G$, it follows that for every edge $\{gP,g'P'\}$ of $\Delta$, the set $\mathcal{N}_{\tau}(gP)\cap \mathcal{N}_{\tau}(g'P')$ is unbounded in $(G,\dist)$. It follows that $\Delta$ is a connected subgraph of $\mathcal{K}_\tau(G,\calp)$, and since both graphs have the same vertex set, we have that $\mathcal{K}_\tau(G,\calp)$ is connected. Since $\calp$ is non-empty,  Proposition~\ref{prop:NetworkChar} implies that $(G,\calp)$ is a $\tau$-network.
\end{proof}

\begin{lemma}\label{lem:dotqsigmatau}
Let $q\colon (G,\calp) \to (H,\calq)$ be a  quasi-isometry of group pairs. For every $\sigma\geq0$ there is $\tau_0\geq0$ such that for every $\tau\geq\tau_0$  the simplicial map $\dot q\colon \mathcal{K}(G,\calp) \to \mathcal{K}(H,\calq)$ restricts to a simplicial map  $\dot q\colon \mathcal{K}_\sigma(G,\calp) \to \mathcal{K}_\tau(H,\calq)$.
\end{lemma}
\begin{proof}
Let us assume that $q\colon (G,\calp) \to (H,\calq)$ is a $(L,C,M)$-quasi-isometry of pairs. Let $\tau_0=L\sigma+C+M$.
Suppose that $\{g_1P_1, \ldots, g_kP_k\}$ is a simplex of $\mathcal{K}_\sigma(G,\calp)$. It follows that  $\bigcap_{i=1}^k \mathcal{N}_\sigma(g_iP_i)$ is unbounded in $(G,\dist_G)$. Hence $\bigcap_{i=1}^k  \mathcal{N}_{L\sigma+C}(q(g_iP_i))$ is unbounded in $(H,\dist_H)$. As a consequence, $\bigcap_{i=1}^k  \mathcal{N}_{\tau}(\dot q(g_iP_i))$ is unbounded for all $\tau\geq \tau_0$. Therefore   $\{\dot q(g_1P_1), \ldots, \dot q (g_kP_k)\}$ is a simplex of $\mathcal{K}_\tau(H,\calq)$.
\end{proof}

The following theorem proves Theorem~\ref{thm:GeomProps}\eqref{item:network}.
\begin{theorem}\label{thm:network}
Being a network is a quasi-isometry invariant of group pairs.  
\end{theorem}

\begin{proof}
Suppose that $q\colon (G,\calp) \to (H,\calq)$ is a $(L,C,M)$-quasi-isometry of pairs and $(G,\calp)$ is a $\sigma$-network. By Lemma~\ref{lem:dotqsigmatau}, there is $\tau\geq0$ and an induced simplicial map  $\dot q\colon \mathcal{K}_\sigma(G,\calp) \to \mathcal{K}_{\tau}(H,\calq)$. By increasing $\tau$ if necessary, assume that $\tau\geq M$.

We claim that $\mathcal{K}_{\tau}(H,\calq)$ is connected. 
Let $hQ,h'Q' \in H/\calq$.  
Proposition~\ref{propF:InducedDot-q-03} implies that   there are $gP,g'P' \in G/\calp$ such that $\Hdist_H(\dot q(gP), hQ)\leq M$ and $\Hdist_H(\dot q(g'P'), h'Q')\leq M$ and therefore $\{\dot q(gP), hQ\}$ and  $\{\dot q(g'P'), h'Q'\}$ are edges of $\mathcal{K}_{\tau}(H,\calq)$.  Since $(G,\calp)$ is a $\tau$-network, Proposition~\ref{prop:NetworkChar} implies that there is an edge-path in $\mathcal{K}_\sigma(G,\calp)$ from $gP$ to $g'P'$. Since $\dot q\colon \mathcal{K}_\sigma(G,\calp) \to \mathcal{K}_{\tau}(H,\calq)$ is a simplicial map, there is a path from $\dot q(gP)$ to $\dot q(g'P')$. By considering the edges $\{q(gP), hQ\}$ and  $\{q(g'P'), h'Q'\}$  of $\mathcal{K}_{\tau}(H,\calq)$, it follows that there is path in $\mathcal{K}_{\tau}(H,\calq)$ from $hQ$ to $h'Q'$. This shows that $\mathcal{K}_{\tau}(H,\calq)$ is connected, and Proposition~\ref{prop:NetworkChar} then implies that $(H,\calq)$ is a $\tau$-network.
\end{proof}

Networks are essential in the notion of \textit{thickness}, which was introduced by Behrstock, Dru\c{t}u, and Mosher in \cite{BDM09}.  We first define thickness for a metric space, and then, as will be most relevant for this paper, define thickness for a group pair.  

\begin{definition}[{\cite[Defintion~7.1]{BDM09}}]
    A collection of metric spaces $\mathcal Y$ is \textit{uniformly thick of order 0} if $\mathcal Y$ is uniformly unconstricted, that is, if there is a constant $c\geq 0$ such that every point in every space $Y\in \mathcal Y$ is at distance at most $c$ from a bi-infinite geodesic in $Y$ and, roughly,  every sequence of spaces $Y_i\in \mathcal Y$ has  an asymptotic cone without a cut point.  See \cite[Definition~3.4]{BDM09} for a precise definition.
    
    The space $X$ is \textit{thick of order at most $n$} with respect to a collection of subspaces $\mathcal Y$ if there is a $\tau\geq 0$ such that $\mathcal Y$ forms a $\tau$--network in $X$ and each $Y\in \mathcal Y$ is thick of order at most $n-1$ when endowed with the subspace metric from $X$.  If $X$ is thick or order at most $n$ but not thick of order at most $n-1$,  we say $X$ is \textit{thick of order $n$}.  If the order of thickness is not important,  we simply say $X$ is \textit{thick}.

    Finally, we say a group pair $(G,\calp)$ is \textit{thick of order $n$} if $(G,d_G)$ is thick of order $n$ with respect to the collection of subspaces $G/\calp$.
\end{definition}

This definition of thickness is also called \textit{metric thickness} in \cite{BDM09}.  Our definition of thickness for a group pair $(G,\calp)$ is slightly stronger than simply saying that $G$ is (metrically) thick of order $n$, because we require the subspaces to be (cosets of) subgroups.  Behrstock, Dru\c{t}u, and Mosher also define the notion of \textit{algebraic thickness} for a group, which requires that the subspaces in the network be subgroups that satisfy certain additional algebraic conditions. 
 These additional conditions are not necessarily preserved by quasi-isometries of pairs, so we do not consider algebraic thickness in this paper.

The following corollary of Theorem~\ref{thm:network} proves Theorem~\ref{thm:GeomProps}\eqref{item:thickness}.
\begin{corollary}\label{cor:thick1}
Being thick of order one is a quasi-isometry invariant of group pairs.
\end{corollary}
\begin{proof}
    Let $q\colon (G,\calp)\to (H,\calq)$ be a quasi-isometry of pairs. Suppose that $(G,\calp)$ is thick of order one.  By the definition of thickness, $G/\calp$ is uniformly thick of order 0, and thus uniformly unconstricted. Since each element of $H/\calq$ is uniformly quasi-isometric to an element of $G/\calp$, it follows that $H/\calq$ is also uniformly thick of order 0; see also \cite[Remark~3.5]{BDM09}.  By Theorem~\ref{thm:network}, $(H,\calq)$ is $\tau$--network for some $\tau\geq 0$.  Therefore, $(H,\calq)$ is thick of order at most one.  As $H$ is quasi-isometric to $G$, and $G$ is not thick of order 0 by assumption, $H$ is thick of order one. 
\end{proof}

A collection of subgroups $\calp$ is qi-characteristic if the collection $G/\calp$ is preserved up  to uniform finite Hausdorff distance by every quasi-isometry of $G$, see~\cite{EMP21-qichar} for a precise definition.
\begin{question}\label{q:qichar}
Let $G$ be a finitely generated group which is thick of order $n$.  Is there a finite qi-characteristic collection of subgroups which respect to which $G$ is thick of order $n$?  
\end{question}

A positive answer to this question is first step towards using induction to prove that the class of group pairs that are thick of order $n$ is geometric.    The following example gives a blueprint for how such an argument might go.

\begin{example}
 Suppose $(G,\calp)$ is thick of order 2, that $G$ and each $P\in\calp$ is finitely generated, and that Question~\ref{q:qichar} has a positive answer.  If $q\colon (G,\calp)\to (H,\calq)$ is a quasi-isometry of pairs, then $(H,\calq)$ is thick of order 2.

   To prove this, note that by Theorem~\ref{thm:network}, the collection $H/\calq$ forms a network in $H$, and thus it suffices to show that each $Q\in\calq$ is thick of order 1.  Let $Q\in\calq$.  Then $Q$ is at finite Hausdorff distance from $q(gP)$ for some $gP\in \calp$.  By \cite[Proposition~2.16]{EMP21-qichar} there is a quasi-isometry $q'\colon P \to Q$.  By assumption,  $P$ is a finitely generated group that is thick of order one. A positive answer to  Question~\ref{q:qichar} thus gives a finite qi-characteristic collection of  subgroups $K_1,\dots, K_m\leq P$ so that $(P,\{K_i\})$ is thick of order one.  Since this collection is qi-characteristic, there is a collection $\{J_i\}$ of subgroups of $Q$ such that $q'\colon (P,\{K_i\})\to (Q,\{J_i\})$ is a quasi-isometry of pairs.  It follows from Corollary~\ref{cor:thick1} that $(Q,\{J_i\})$ is thick of order 1, which completes the argument.
 \end{example}




\section{Two dimensional right-angled Artin groups}\label{sec:RAAGs} 

Let $G$ be a right-angled Artin group whose defining graph $\Gamma$ has no vertices of valence $<2$ and no triangles.  Abusing notation, we identify the set of vertices of $\Gamma$ with the standard generating set of $G$. A subgroup of $G$ generated by a collection of standard generators is called a \emph{standard subgroup}.
For a standard generator $a$, let \[ \mathsf{Link}(a)=\{b\mid \text{$b$ is a standard generator, $a\neq b$, and $a$ and $b$ commute}\} \]
and 
\[ \mathsf{Star}(a)=\mathsf{Link}(a)\cup \{a\}.\]
A subgroup of the form $\langle \mathsf{Star}(a) \rangle$ is called a \emph{standard star subgroup.}

In this section, we will focus on the  collection  $\calp$  of maximal standard abelian subgroups, so that the elements in $\calp$ are subgroups generated by the vertices of maximal cliques in $\Gamma$.  Our assumption that $\Gamma$ has no triangles and no vertices of valence $<2$ implies that maximal cliques in $\Gamma$ are edges, and so maximal standard abelian subgroups are  $\mathbb Z^2$--subgroups of the form 
\[ \calp=\{\langle a,b \rangle  \mid \text{$a$ and $b$ are commuting standard generators and $a\neq b$} \}.\]

 The \emph{extension graph} $E(\Gamma)$ of Kim and Koberda~\cite{KK13} is the $G$-graph with vertex set $\{v^g\mid v\in V(\Gamma),\ g\in G\}$ such that two distinct vertices $u^g$ and $v^h$ are adjacent if and only if they commute in $G$. In this section we prove the following result.

\begin{theorem}
\label{thm:ExtensionComplex}
    Suppose $\Gamma$ is connected, has no vertices of valence $<2$, and no triangles. Then there is a quasi-isometry between $\mathcal{K}(G,\calp)$ and the extension graph $E(\Gamma)$.
\end{theorem} 

The proof of the theorem has three steps, each done in a separate subsection. 
First, we prove that the commensurated core   of $(G,\calp)$ is 
$\calq=\{\langle \mathsf{Star}(a)\rangle \mid a \textrm{ is a standard generator}\}$. Second, we show that the coset intersection complexes $\mathcal{K}(G,\calp)$ and $\mathcal{K}(G,\calq)$ are quasi-isometric. Finally, we conclude by proving that $ \mathcal{K}(G,\calq)$ and the coned-off Cayley graph $\widehat{\Gamma}(G,\calq)$ are quasi-isometric; see Section~\ref{sec:coned-off}.  Since $\langle \mathsf{Star}(a)\rangle$ contains the generator $a$, the collection $\calq$ generates the group $G$, and so we may assume the relative generating set is the empty set.  Theorem~\ref{thm:ExtensionComplex} then follows by a result of Kim and Koberda \cite[Theorem 15]{KK14} stating that $E(\Gamma)$ is quasi-isometric to   $\widehat{\Gamma}(G, \mathcal{Q})$ .  

\subsection{Commensurated cores in right-angled Artin groups} 

\begin{proposition} \label{prop:calQisStar}
Suppose $\Gamma$ has no vertices of valence $<2$ and no triangles.
The $G$-set of maximal simplices of $\mathcal{K}(G,\calp)$ is isomorphic to the $G$-set of left cosets  $G/\mathcal{Q}$ where 
\[ \mathcal{Q}=\{\langle \mathsf{Star}(a)\rangle\mid \text{$a$ is a standard generator}\}.\]
In particular, $\calq$  is the commensurated core of $(G,\calp)$.
\end{proposition}

The proposition is an immediate consequence of the following lemma.

\begin{lemma}\label{lem:RAAG-maxSimplex} Suppose $\Gamma$ has no vertices of valence $<2$ and no triangles.
If $\cals$ is a maximal simplex of $\mathcal{K}(G,\calp)$, then there is $g\in G$ and a standard generator $a$ such that the following hold.
    \begin{enumerate}
        \item $\cals=\{ gw\langle a,c \rangle \mid w\in\langle \mathsf{Star}(a)\rangle \text{ and $c\in\mathsf{Link}(a)$} \}$ .
        \item For any two distinct left cosets $g_1P_1,g_2P_2 \in \cals$, $g_1P_1g_1^{-1}\cap g_2P_2g_2^{-1}= g\langle a\rangle g^{-1}$. In particular, the pointwise stabilizer of $\cals$ is $g\langle a\rangle g^{-1}$.
        \item the setwise stabilizer of $\cals$ is $\comm_G(g\langle a \rangle g^{-1})$,  and
\[ \comm_G(g\langle a \rangle g^{-1}) =g\langle \mathsf{Star}(a)\rangle g^{-1}.\] 
        \item $\cals$ is infinite dimensional.
    \end{enumerate}

\end{lemma}

\begin{proof}
Up to translating $\cals$ by an element $g$, assume that $\langle a,b\rangle \in \cals$ where $a,b$ are commuting standard generators.  Observe that $\cals$ has more than one element. Indeed, since $a$ has valence $\geq 2$ in $\Gamma$, there is $c\in \mathsf{Link}(a)$ such that $b\neq c$.  The cosets $\langle a,b\rangle$ and $\langle a,c\rangle$ form a $1$-dimensional simplex, and hence if $\cals$ has a single element then it is not maximal.

Let $wP = w\langle x,y\rangle$ be an element of $\cals$  different than $\langle a,b\rangle$. The intersection $\langle a,b \rangle \cap w\langle x,y \rangle w^{-1} $ is an infinite parabolic subgroup, and the only possibilities  are $\langle a\rangle$, $\langle b\rangle$ or $\langle a,b \rangle$; see~\cite[Corollary~2.5]{DKR07}.  Let us argue that
 $\langle a,b \rangle \cap w\langle x,y \rangle w^{-1}$ is $\langle a\rangle$ or $\langle b\rangle$. 
 Suppose $\langle a,b\rangle \cap w\langle x,y\rangle w^{-1} = \langle a,b\rangle$. Then $\{x,y\}=\{a,b\}$ by \cite[Lemma~2.7]{DKR07}, and therefore $w$ is in the normalizer of $\langle a,b\rangle$. 
  In a right-angled Artin group, centralizers of parabolic subgroups coincide with normalizers~\cite{Go03}.  By the Centralizer Theorem~\cite{Se85}, the centralizer $C(u)$ of a standard generator $u$ is the subgroup generated by $\mathsf{Star}(u)$, and therefore the centralizer of $\langle a,b \rangle$ is $\langle\mathsf{Star}(a)\rangle \cap \langle\mathsf{Star}(b)\rangle$. By \cite[Lemma~2.7]{DKR07} and since $\Gamma$ has no triangles, we have  $C(\langle a,b \rangle) = \langle\mathsf{Star}(a)\rangle \cap \langle\mathsf{Star}(b)\rangle = \langle \mathsf{Star}(a) \cap \mathsf{Star}(b) \rangle = \langle a,b \rangle$. Hence $w\in \langle a,b\rangle$, showing that $\langle a,b\rangle = w\langle x,y\rangle$, contrary to our assumption.  

Since $\cals$ is a simplex, we can assume without loss of generality that  $\langle a,b\rangle \cap wPw^{-1} = \langle a \rangle$ for \textit{every} $wP\in\cals$ different than  $\langle a, b\rangle$. Therefore the pointwise stabilizer of $\cals$ satisfies $K(\cals)=\langle a \rangle$ and, by maximality of $\cals$, \begin{align*} \cals&=\{ wP\mid w\in G,\  P\in\calp,\  a\in wPw^{-1}  \}.  
\end{align*}

By \cite[Lemma~2.7]{DKR07}, if  $a,x,y$ are standard generators,  $h\in G$, and $a\in h\langle x,y\rangle h^{-1}$, then $a\in \{x,y\}$ and $h$ belongs to the product set $C(a)\cdot  \langle x,y\rangle$.  In particular, $h\langle x,y\rangle h^{-1} = w\langle x,y\rangle w^{-1}$ for some $w\in C(a)=\langle \mathsf{Star}(a)\rangle$. Therefore
\[ \cals = \{w\langle a,c\rangle \mid w\in\langle \mathsf{Star}(a)\rangle \text{ and } c\in\mathsf{Link}(a) \}.\]
The expression above shows  that $\cals$ is infinite. More explicitly, since $\Gamma$ has no vertices of valence $<2$, there are distinct standard generators $b$ and $c$ in $\mathsf{Link}(a)$. Then, for every $n\in\ZZ$, the coset $b^n\langle a,c \rangle$ is an element of $\cals$. Hence $\cals$ is infinite. 

It is also a consequence of~\cite[Lemma~2.7]{DKR07} that the  normalizer and commensurator of  $\langle a\rangle$ coincide, and, as stated above, they are also equal the centralizer $C(a)= \langle \mathsf{Star}(a) \rangle$. 
Hence, by Lemma~\ref{lem:SetWiseStabilizers} and the Centralizer theorem, the setwise stabilizer of $\cals$ is $\comm_G(K(S))=\comm_G(\langle a\rangle) = \langle \mathsf{Star}(a) \rangle$.
\end{proof}
   
\subsection{Coset intersection complex of the commensurated core}\label{sec:RAAGsCommCore}

In general, it is not known how replacing a collection $\calp$ of subgroups with its commensurated core will affect the quasi-isometry type of $\mc K(G,\calp)$.  However, for the right-angled Artin groups considered in this section, replacing the collection $\calp$ of maximal standard abelian subgroups with its commensurated core $\calq$ actually preserves the quasi-isometry type.  It remains to explore this relationship in general.

Recall from Proposition~\ref{prop:calQisStar} that 
\[
\calq=\{\langle \mathsf{Star}(a)\rangle \mid a \textrm{ is a standard generator}\}.
\]

 \begin{proposition}\label{prop:RAAG-P-Q}
    Suppose $\Gamma$ is connected, has no vertices of valence $<2$, and no triangles. Then $\mathcal{K}(G,\calp)$ is connected, and there is a quasi-isometry $\mathcal{K}(G,\calp) \to \mathcal{K}(G,\calq)$ that is $G$-equivariant.
\end{proposition}

  Before proving the proposition, we gather two lemmas about the structure of $\mc K(G,\calp)$.
\begin{lemma}\label{lem:RAAGConnected}
   If $\Gamma$ is connected, then  $\mathcal{K}(G,\calp)$ is connected.
\end{lemma}
\begin{proof}
 Let $\Delta$ be the graph with vertex set $\calp$ and such that two subgroups are adjacent if they have infinite intersection. Equivalently, $\Delta$ is the graph with vertex set the set of edges of $\Gamma$ and such that two edges are adjacent if they have an endpoint in common. Since $\Gamma$ is connected, $\Delta$ is  connected. Moreover $\Delta$ has a natural embedding into $\mathcal{K}(G,\calp)$ such that $\Delta$ contains representatives of all $G$-orbits of vertices of  $\mathcal{K}(G,\calp)$. Observe that for any standard generator $a$, the subgraphs $\Delta$ and $a\Delta$ of $\mathcal{K}(G,\calp)$ intersect in at least one vertex, namely, a vertex corresponding to subgroup in $\calp$ containing $a$. It follows that for any $g\in G$, there is a path from $\Delta$ to $g\Delta$ in  $\mathcal{K}(G,\calp)$. Therefore $\mathcal{K}(G,\calp)$  
 is connected.
\end{proof}

\begin{lemma}\label{lem:IntofSxs}
If $\Gamma$ has no vertices of valence $<2$ and no triangles, then any two distinct maximal simplices of $\mathcal{K}(G,\calp)$ intersect in at most one vertex.
\end{lemma}
\begin{proof}
 Let $\cals$ and $\cals'$ be two distinct maximal simplices. Suppose without loss of generality that $K(\cals)=\langle a \rangle$ and $K(\cals')=w\langle b\rangle w^{-1}$ for some $w\in G$. Since $\cals \neq \cals'$, we have that $\langle a \rangle \cap w\langle b\rangle w^{-1}$ is trivial by Lemma~\ref{lem:RAAG-maxSimplex}. Moreover, if $\cals \cap \cals'$ contains two distinct left cosets $g_1P_1$ and $g_2P_2$, then again applying Lemma~\ref{lem:RAAG-maxSimplex}, we obtain that $\langle a \rangle=g_1P_1g_1^{-1}\cap g_2P_2g_2^{-1} = w \langle b \rangle w^{-1}$, which is impossible.     
\end{proof}

We are now ready to prove Proposition~\ref{prop:RAAG-P-Q}.

\begin{proof}[Proof of Proposition~\ref{prop:RAAG-P-Q}]
We begin by defining two auxiliary quasi-isometric $G$-graphs, $\mc E$ and its barycentric subdivision $\mc E'$.  We will show that show that $\mc K(G,\calp)$ is quasi-isometric to $\mc E'$ and that $\mc K(G,\calq)$ is quasi-isometric to $\mc E$.  All of these quasi-isometries will be $G$-equivariant, which will prove the result.

Let $\mc E$ be the $G$-graph with vertex set in bijection with the set of maximal simplices in $\mc K(G,\calp)$ and an edge between two vertices if the corresponding maximal simplices share a common vertex.   Let $\mc E'$ be the barycentric subdivision of $\mc E$, so that $\mc E$ is $G$-equivariantly quasi-isometric to $\mc E'$.

We first show that $\mc K(G,\calp)$ is quasi-isometric to $\mc E'$.   
 Since $\Gamma$ has no vertices of valence $<2$, the $G$-sets $G/\calp$ and $G/\calq$ are disjoint.  By Proposition~\ref{prop:calQisStar}, we can equivalently define $\mc E'$ to be the bipartite $G$-graph with vertex set $G/\calp\sqcup G/\calq$ and edge set given by inclusion of cosets.  Using this alternative description of $\mc E'$ and the fact that maximal simplices of $\mc K(G,\calp)$ are in bijection with $G/\calq$ by Proposition~\ref{prop:calQisStar}, we see that $\mc E'$ is formed from the 1-skeleton of $\mc K(G,\calp)$ by coning off each maximal simplex with a vertex and then removing the original edges of $\mc K(G,\calp)$.  Since maximal simplices have finite diameter, coning off each one does not change the quasi-isometry type.  Deleting the edges of $\mc K(G,\calp)$ also does not change the quasi-isometry type, since there is a 2-path in $\mc E'$ between the endpoints of each edge removed.

We next show that $\mc K(G,\calq)$ is quasi-isometric to $\mc E$.  
 By Lemmas~\ref{lem:RAAG-maxSimplex} and \ref{lem:IntofSxs}, the graph $\mc E$ can be equivalently defined as the graph with vertex set $G/\calq$ and such that $g_1Q_1, g_2Q_2$ are adjacent if and only if $Q_1^{g_1}\cap Q_2^{g_2}$ has rank two.  To prove $\mc K(G,\calq)$ and $\mc E$ are quasi-isometric, first note that both have the same vertex set and every edge of $\mathcal E$ is an edge of $\mathcal{K}(G,\calq)$. For distinct vertices $g_1Q_1, g_2Q_2\in G/\calq$, the intersection 
 $Q_1^{g_1}\cap Q_2^{g_2}$ is a free abelian group of rank at most two. If the rank is two, then this intersection is an edge of $\mathcal E$. 
 If
$Q_1^{g_1}\cap Q_2^{g_2}$ has rank one, then $g_1Q_2$ and $g_2Q_2$ are vertices at distance two in $\mathcal{E}$. 
Therefore  the inclusion $\mathcal{E} \hookrightarrow \mathcal{K}(G,\calq)$ is a $G$-equivariant quasi-isometry.    
\end{proof}

\subsection{The extension graph}

In this subsection, we prove that $\mathcal{K}(G,\calq)$ and the coned-off Cayley graph $\widehat{\Gamma}(G,\calq)$ are quasi-isometric.  We then  conclude with the proof of Theorem~\ref{thm:ExtensionComplex}. 


The following proposition is an application of~\cite[Lemma~2.7]{DKR07} which describes the intersection of parabolic subgroups in right-angled Artin groups. Recall that $\widetriangle\Gamma(G,\calq)$ denotes the extended coned-off Cayley graph; see Section~\ref{sec:coned-off}.

\begin{proposition}\label{prop:RAAGsConedOff}
If $\Gamma$ has no vertices of valence $<2$ and no triangles, then the inclusion $\widehat\Gamma(G,\calq) \hookrightarrow \widetriangle\Gamma (G,\calq)$ is a quasi-isometry. 
\end{proposition} 
\begin{proof}
    Up to translating, any edge of $\mathcal{K}(G,\calq)$ corresponds to a pair of left cosets of the form $\langle \mathsf{Star}(a) \rangle$ and $g\langle \mathsf{Star}(b) \rangle$ such that the corresponding subgroups have infinite intersection, that is, $\langle \mathsf{Star}(a) \rangle \cap  \langle \mathsf{Star}(b) \rangle^g$ is not trivial.
Duncan, Kazachkov, and Remeslennikov \cite[Lemma~2.7]{DKR07} provide a subset $T$ of the standard generators and a decomposition $g=g_1dg_2$ such that 
\[\langle \mathsf{Star}(a) \rangle \cap  \langle \mathsf{Star}(b) \rangle^g = \langle \mathsf{Star}(a) \cap \mathsf{Star}(b) \cap T \rangle^{g_1},\] where  $g_1\in \langle \mathsf{Star}(a)\rangle$,
$g_2\in \langle \mathsf{Star}(b)\rangle$,    and $d\in \langle   \mathsf{Star}(T)$. 
Since $\langle \mathsf{Star}(a) \rangle \cap  \langle \mathsf{Star}(b) \rangle^g$ is not trivial, there exists  an element $x\in T\cap \mathsf{Star}(a) \cap \mathsf{Star}(b)$.  Since $x\in T$, we have $\mathsf{Star}(T)\subseteq \mathsf{Star}(x)$, and so $d\in \langle \mathsf{Star}(x)\rangle$.  It follows that the sequence $\langle \mathsf{Star}(a)\rangle$, $x$, $\langle \mathsf{Star}(x)\rangle$, $d$, $d\langle\mathsf{Star}(b)\rangle = dg_2 \langle\mathsf{Star}(b)\rangle$ is a path of length four in   $\widehat{\Gamma}(G, \mathcal{Q})$. Therefore the left cosets    
$\langle \mathsf{Star}(a) \rangle = g_1\langle \mathsf{Star}(a) \rangle$ and $g\langle \mathsf{Star}(b) \rangle=g_1dg_2 \langle \mathsf{Star}(b) \rangle$ are also at distance at most four in $\widehat{\Gamma}(G, \mathcal{Q})$. Hence any pair of adjacent vertices in $\mathcal{K}(G,\calq)$ are at distance at most four in $\widehat{\Gamma}(G, \mathcal{Q})$, and therefore 
$\widehat{\Gamma}(G, \mathcal{Q})   \hookrightarrow \widetriangle{\Gamma}(G, \mathcal{Q})$  is a quasi-isometry.
\end{proof}

\begin{proposition}\label{prop:RAAG-hatGamma-Q}
 If $\Gamma$ has no vertices of valence $<2$ and no triangles, then  $\mathcal{K}(G,\calq)$ and $\widehat\Gamma(G,\calq)$ are quasi-isometric.    
\end{proposition}
\begin{proof}
By  Proposition~\ref{prop:RAAG-P-Q}, the coset intersection complex   $\mathcal{K}(G,\calq)$ is connected. Then  Proposition~\ref{prop:QI-coned-off} implies that    $\mathcal{K}(G,\calq)$ and $\widetriangle\Gamma(G,\calq)$ are quasi-isometric. The conclusion follows from Proposition~\ref{prop:RAAGsConedOff}.
\end{proof}

\begin{proof}[Proof of Theorem~\ref{thm:ExtensionComplex}] 
The complexes $\mathcal{K}(G,\calp)$ and $\mathcal{K}(G,\calq)$ are quasi-isometric by Proposition~\ref{prop:RAAG-P-Q}, and $\mathcal{K}(G,\calq)$ and $\widehat{\Gamma}(G,\calq)$ are quasi-isometric by Proposition~\ref{prop:RAAG-hatGamma-Q}. Moreover, $\widehat{\Gamma}(G,\calq)$ is quasi-isometric to 
$E(\Gamma)$ by~\cite[Theorem 15]{KK14}. Therefore, $\mathcal{K}(G,\calp)$ and $E(\Gamma)$ are quasi-isometric.    
\end{proof}


\subsection{Other related complexes}

There are a number of $G$-complexes that have been introduced in the literature in order to understand algebraic and geometric properties of a right-angled Artin group $G$ with defining graph $\Gamma$. Some of these complexes relate to the coset intersection  complex $\mathcal{K}(G,\calp)$ of $G$ with respect to the collection $\calp$ of maximal standard abelian subgroups.

One of the spaces that concerns us arises from the \emph{Salvetti complex $S(\Gamma)$}~\cite{ChDA95}. This is an Eilenberg-MacLane space of $G$ with the structure of a locally CAT(0) cube complex whose $2$-skeleton coincides with the usual presentation complex of $G$. The universal cover $X(\Gamma)$ of $S(\Gamma)$ is a  CAT(0) cube complex.

The  \emph{contact graph $\mathcal{C}(\Gamma)$},  introduced by Hagen~\cite{HagenThesis}, is the graph with vertex set    in one-to-one correspondence with the hyperplanes of $X(\Gamma)$, where distinct vertices  are adjacent if the carriers of the corresponding hyperplanes intersect.
 Hagen showed that the contact graph is quasi-isometric to a tree~\cite{Ha14}.
Kim and Koberda remark that if $\Gamma$ has no isolated vertices, then the extension graph $E(\Gamma)$ is quasi-isometric to the contact graph  $\mathcal{C}(\Gamma)$; see~\cite[Section 7]{KK14}. We note that Huang has found examples of commensurable RAAGs with different extension graphs, and a pair of non-quasi-isometric RAAGs with isomorphic extensions graphs~\cite[Sec. 5.3]{Hua16}.

In~\cite{Hu17}, Huang introduced the \emph{extension complex  $E^\blacktriangle(\Gamma)$} via two equivalent definitions. First, as the flag complex of the extension graph; and alternatively as the flag complex of the graph whose vertices are in one-to-one correspondence with the parallelism classes of standard geodesics in the universal cover $X(\Gamma)$ of the Salvetti complex and such that two distinct vertices are connected by an edge if and only if there are geodesic representatives that span a 2-flat. Huang remarks that $E^\blacktriangle(\Gamma)$ can be viewed as a \emph{simplified Tits boundary} for $X(\Gamma)$. In the class of RAAGs with transvection-free outer automorphim groups, a quasi-isometry between two groups induces an isomorphism between their extension complexes~\cite[Lemma 4.5]{Hu17}, an statement that resembles Theorem~~\ref{thm:SimplicialMap} via Theorem~\ref{thm:RAAG}. 

Quasi-median graphs introduced by Mulder in the 1980's turned out to be  generalizations of 1-skeletons of CAT(0) cubical complexes~\cite{Mu80, Ge17}. This perspective has been explored in the context of geometric group theory by Genevois~\cite{Ge17}, who shows that quasi-median graphs admit a notion of hyperplanes  generalizing  the theory of hyperplanes in nonpositively curved cube complexes. In this context, for a quasi-median graph $M$, the  \emph{transversality graph~\cite{GeMa19}} (called the \emph{crossing graph} in~\cite{Ge17}) is the complex  $T(M)$ with vertex set in one-to-one correspondence with the collection of hyperplanes of $M$ and whose simplices correspond to finite collections of pairwise transverse hyperplanes; see~\cite{Ge17} for details. The Cayley graph $\Gamma(G,\bigcup \mathcal{O})$ of a RAAG $G$ with respect to the generating set $\bigcup\mathcal{O}=\bigcup \{\langle a \rangle \mid \text{$a$ is an standard generator} \}$ is a quasi-median graph. The  transversality graph $T(\Gamma)$ of $\Gamma(G,\bigcup\mathcal O)$ is  isomorphic to the extension graph; see~\cite[Fact 8.25 and Proof of Cor. 8.49]{Ge17}. 

Putting together Theorem~\ref{thm:ExtensionComplex}, Proposition~\ref{prop:RAAG-P-Q}, and the results in the literature described above, we obtain the following.

\begin{corollary}
  Suppose $\Gamma$ is connected, has no vertices of valence $<2$, and no triangles. The contact graph $\mathcal{C}(\Gamma)$, the extension graph $E(\Gamma)$, the transversality graph $T(\Gamma)$,  the coset intersection complexes $\mathcal{K}(G,\calp)$ and $\mathcal{K}(G,\calq)$, and the coned-off Cayley graph $\widehat\Gamma(G,\calq)$ are  pairwise quasi-isometric, and they are all quasi-isometric to a tree.    
\end{corollary}

We expect the  above result will hold  for  more general right-angled Artin groups.  We also expect that there are relationships between the homotopy types of the coset intersection complex,   the extension complex, and the flag completion of the transversality graph.  We explore some of these relationships in forthcoming work with Genevois.

We end our discussion of right-angled Artin groups with several general questions about the associated coset intersection complex.

\begin{question}\label{que:SimplyConnected}
  Let $G$ be a 2-dimensional right-angled Artin group with defining graph $\Gamma$, and $\mathcal P$  the collection of  maximal standard abelian subgroups.
  Suppose the induced flag complex $\Gamma^\blacktriangle$ is simply-connected. 
  Is $\mathcal{K}(G, \calp)$ simply connected?
\end{question}

\begin{question}
Under the assumptions of Question~\ref{que:SimplyConnected}, is $\mc K(G,\calp)$ contractible? 
\end{question}

More generally, one can ask the following. 

\begin{question}
    Let $G$ be a 2-dimensional RAAG and $\mathcal P$  the collection of maximal standard abelian subgroups.  Let $\mathcal{F}$ be the family generated by $\calp$ that is closed under conjugation and subgroups. In which cases is $\mc K(G,\calp)$ a classifying space for $G$ with respect to $\mathcal{F}$? 
\end{question}

\bibliographystyle{alpha} \bibliography{xbib}

\end{document}